\documentclass[tbtags,reqno]{amsart}
\usepackage{mathrsfs}
\usepackage{latexsym}
\usepackage{amsthm,amsopn,amsfonts,amssymb,epsfig,color,pstricks,pb-diagram,mathdots,stmaryrd}
\usepackage[active]{srcltx}
\def\fr{\frac}
\edef\savecatcodeat{\the\catcode`@}
\catcode`\@=11

\def\tb@ifSpecChars#1#2{#1}
\def\tb@ifNoSpecChars#1#2{#2}

\def\tableau{%
  \bgroup
  \@ifstar{\let\Tif\tb@ifNoSpecChars\tb@tableauB}
          {\let\Tif\tb@ifSpecChars\tb@tableauB}}

\def\tb@tableauB{
  \@ifnextchar[{\tb@tableauC}{\tb@tableauC[]}}

\def\tb@tableauC[#1]{\hbox\bgroup%
    \let\\=\cr
    \def\bl{\global\let\tbcellF\tb@cellNF}%
    \def\tf{\global\let\tbcellF\tb@cellH}
    \dimen2=\ht\strutbox \advance\dimen2 by\dp\strutbox%
    \ifx\baselinestretch\undefined\relax%
    \else%
       \dimen0=100sp \dimen0=\baselinestretch\dimen0%
       \dimen2=100\dimen2 \divide\dimen2 by\dimen0%
    \fi%
    \let\tpos\tb@vcenter
    \tb@initYoung
    \tb@options#1\eoo
    \let\arrow\tb@arrow%
    \dimen0=\Tscale\dimen2%
    \dimen1=\dimen0 \advance\dimen1 by \tb@fframe%
    \lineskip=0pt\baselineskip=0pt
    \def\tb@nothing{}%
    \def\endcellno{$\rss\egroup\bss\egroup}
    \def\endcell{\endcellno\kern-\dimen0}
    \def\begincell{\vbox to\dimen0\bgroup\vss\hbox to\dimen0\bgroup\hss$}%
    \let\overlay\tb@overlay%
    \let\fl\tb@fl%
    \let\lss\hss\let\rss\hss\let\tss\vss\let\bss\vss
    \def\mkcell##1{
        \let\tbcellF\tb@cellD
        \def\tb@cellarg{##1}
        \ifx\tb@cellarg\tb@nothing\let\tb@cellarg\tb@cellE\fi%
        \begincell\tb@cellarg\endcellno
        \tbcellF}
    \let\savecellF\tbcellF
     \Tif{\catcode`,=4\catcode`|=\active}{}\tb@tableauD}%

\let\tb@savetableauD\tableauD
{
    \catcode`|=\active \catcode`*=\active \catcode`~=\active%
    \catcode`@=\active
\gdef\tableauD#1{%
  \Tif{
    \mathcode`|="8000 \mathcode`*="8000%
    \mathcode`~="8000 \mathcode`@="8000%
    \def@{\bullet}%
    \let|\cr
    \let*\tf
    \let~\sk
  }{}%
  \tpos{\tabskip=0pt\halign{&\mkcell{##}\cr#1\crcr}}%
  \global\let\tbcellF\savecellF
  \egroup
  \egroup}
}
\let\tb@tableauD\tableauD
\let\tableauD\tb@savetableauD
\let\tb@savetableauD\undefined

\def\tb@options#1{\ifx#1\eoo\relax\else\tb@option#1\expandafter\tb@options\fi}

\def\tb@option#1{%
  \if#1t\let\tpos\tb@vtop\fi
  \if#1c\let\tpos\tb@vcenter\fi
  \if#1b\let\tpos\vbox\fi
  \if#1F\tb@initFerrers\fi
  \if#1Y\tb@initYoung\fi
  \if#1s\tb@initSmall\fi
  \if#1m\tb@initMedium\fi
  \if#1l\tb@initLarge\fi
  \if#1p\tb@initPartition\fi
  \if#1a\tb@initArrow\fi
}

\def\tb@vcenter#1{\ifmmode\vcenter{#1}\else$\vcenter{#1}$\fi}

\def\tb@vtop#1{\hbox{\raise\ht\strutbox\hbox{\lower\dimen0\vtop{#1}}}}

\def\tb@initPartition{\def\Tscale{.3}}
\def\tb@initSmall{\def\Tscale{1}}
\def\tb@initMedium{\def\Tscale{2}}
\def\tb@initLarge{\def\Tscale{3}}

\def\tb@initArrow{\dimen2=1.25em}

\def\tb@initYoung{%
  \def\tb@cellE{}
  \let\tb@cellD\tb@cellN
  \def\sk{\global\let\tbcellF\tb@cellNF}}
\def\tb@initFerrers{%
  \def\tb@cellE{\bullet}
  \let\tb@cellD\tb@cellNF
  \def\sk{\bullet}}

\tb@initMedium

\def\tb@sframe#1{%
  \vbox to0pt{
    \vss
    \hbox to0pt{%
      \hss
      \vbox to\dimen1{
        \hrule depth #1 height0pt
        \vss
        \hbox to\dimen1{
          \vrule width #1 height\dimen1
          \hss
          \vrule width #1
          }%
        \vss
        \hrule height #1 depth 0in
        }%
      \kern-\tb@hframe
      }%
    \kern-\tb@hframe}}

\def\tb@hframe{.2pt}\def\tb@fframe{.4pt}\def\tb@bframe{2pt}
\def\tb@cellH{\tb@sframe{\tb@bframe}}       
\def\tb@cellNF{}                            
\def\tb@cellN{\tb@sframe{\tb@fframe}}       
\let\tbcellF\tb@cellN                       

\def\tb@rpad{1pt}
\def\tb@lpad{1pt}
\def\tb@tpad{1.8pt}
\def\tb@bpad{1.8pt}

\def\tb@overlay{\endcell\@ifnextchar[{\tb@overlaya}{\begincell}}
\def\tb@overlaya[#1]{\vbox to\dimen0\bgroup%
  \tb@overlayoptions#1\eoo%
  \tss\hbox to\dimen0\bgroup\lss$}
\def\tb@overlayoptions#1{\ifx#1\eoo\relax\else\tb@overlayoption#1\expandafter\tb@overlayoptions\fi}

\def\tb@overlayoption#1{
  \if#1t\def\tss{\vskip\tb@tpad}\let\bss\vss\fi
  \if#1c\let\tss\vss\let\bss\vss\fi
  \if#1b\def\bss{\vskip\tb@bpad}\let\tss\vss\fi
  \if#1l\def\lss{\hskip\tb@lpad}\let\rss\hss\fi
  \if#1m\let\lss\hss\let\rss\hss\fi
  \if#1r\def\rss{\hskip\tb@rpad}\let\lss\hss\fi
}

\def\tb@fl{\endcell\begincell\vrule depth 0pt width \dimen0 height \dimen0 \endcell\begincell}

\@ifundefined{diagram}{}{
\def\tb@arrowpad{.5}

\newoptcommand{\tb@arrow}{\@ne}[2]{%
  \endcell
   \begingroup%
   \let\dg@getnodesize\tb@getnodesize
   \dg@USERSIZE=#1\relax%
   \ifnum\dg@USERSIZE<\@ne \dg@USERSIZE=\@ne \fi%
   \dg@parse{#2}%
   \dg@label{\tb@draw{#1}{#2}}}

\def\tb@getnodesize#1#2#3#4#5{\dimen3=\tb@arrowpad\dimen2 #4=\dimen3 #5=\dimen3\relax}
\def\tb@getnodesize#1#2#3#4#5{\ifnum#2=0\ifnum#3=0\tb@getnodesizetail{#4}{#5}\else\tb@getnodesizehead{#4}{#5}\fi\else\tb@getnodesizehead{#4}{#5}\fi}
\def\tb@getnodesizetail#1#2{\dimen3=.5\dimen2 #1=\dimen3 #2=\dimen3}
\def\tb@getnodesizehead#1#2{\dimen3=.5\dimen2 #1=\dimen3 #2=\dimen3}

\def\tb@draw#1#2#3#4{%
        \dg@X=0\dg@Y=0\dg@XGRID=1\dg@YGRID=1\unitlength=.001\dimen0%
        \dg@LBLOFF=\dgLABELOFFSET \divide\dg@LBLOFF\unitlength%
        \dg@drawcalc
        \begincell
        \let\lams@arrow\tb@lams@arrow
        \begin{picture}(0,0)\begingroup\dg@draw{#1}{#2}{#3}{#4}\end{picture}%
        \endcell
        \endgroup
        \begincell}
}

\def\tb@lams@arrow#1#2{%
 \lams@firstx\z@\lams@firsty\z@
 \lams@lastx#1\relax\lams@lasty#2\relax
 \lams@center\z@
 \N@false\E@false\H@false\V@false
 \ifdim\lams@lastx>\z@\E@true\fi
 \ifdim\lams@lastx=\z@\V@true\fi
 \ifdim\lams@lasty>\z@\N@true\fi
 \ifdim\lams@lasty=\z@\H@true\fi
 \NESW@false
 \ifN@\ifE@\NESW@true\fi\else\ifE@\else\NESW@true\fi\fi
 \ifH@\else\ifV@\else
  \lams@slope
  \ifnum\lams@tani>\lams@tanii
   \lams@ht\ten@\p@\lams@wd\ten@\p@
   \multiply\lams@wd\lams@tanii\divide\lams@wd\lams@tani
  \else
   \lams@wd\ten@\p@\lams@ht\ten@\p@
   \divide\lams@ht\lams@tanii\multiply\lams@ht\lams@tani
  \fi
 \fi\fi
 \ifH@  \lams@harrow
 \else\ifV@ \lams@varrow
 \else \lams@darrow
 \fi\fi
}

\catcode`\@=\savecatcodeat
\let\savecatcodeat\undefined

\numberwithin{equation}{section}

\makeatletter
\renewcommand{\subsubsection}{\@startsection
{subsubsection} {3} {0mm} {\baselineskip} {-0.5\baselineskip} {\normalfont\normalsize\bfseries}} \makeatother

\def\cal L{{\mathcal L}}

\def \part {\vdash}

\newcommand{\cercle}[1]{\ensuremath{\setlength{\unitlength}{1ex}\begin{picture}(2.8,2.8)\put(1.5,1.2){\circle{2.8}\makebox(-5.6,0){#1}}\end{picture}}}

\theoremstyle{plain}
\newtheorem{definition}{Definition} 
\newtheorem{theorem}[definition]{Theorem}
\newtheorem{example}[definition]{Example}
\newtheorem{lemma}[definition]{Lemma}
\newtheorem{corollary}[definition]{Corollary}
\newtheorem{proposition}[definition]{Proposition}

\newtheorem{remark}[definition]{Remark}

\theoremstyle{definition}

\numberwithin{equation}{section}

\title{The $m$-symmetric Macdonald polynomials}
\begin{document}

\author{Manuel Concha}
\address{Instituto de Matem\'aticas, Universidad de
Talca, 2 Norte 685, Talca, Chile.}
\email{manuel.concha@utalca.cl}
\author{Luc Lapointe}
\address{Instituto de Matem\'aticas, Universidad de
Talca, 2 Norte 685, Talca, Chile.}
\email{llapointe@utalca.cl}

\begin{abstract}
The $m$-symmetric Macdonald polynomials form a basis of the space of polynomials that are
symmetric in the variables $x_{m+1},x_{m+2},\dots$ (while having no special symmetry in the variables $x_1,\dots,x_m$).
We establish in this article the fundamental properties of the $m$-symmetric Macdonald polynomials.  These include among other things the orthogonality with respect to a natural scalar product, as well as formulas for the squared norm, the evaluation, and the inclusion.  We also obtain a Cauchy-type identity for the $m$-symmetric Macdonald polynomials which specializes to the known Cauchy-type identity for the non-symmetric Macdonald polynomials.
\end{abstract}


\thanks{Funding: this work was supported by
the Fondo Nacional de Desarrollo Cient\'{\i}fico y Tecnol\'ogico de Chile (FONDECYT) Regular Grant \#1210688 and the Beca de Doctorado Nacional ANID \#21201319.}

\maketitle

\section{Introduction}
A version of the Macdonald polynomials forming a basis of the space $R_m$ of polynomials that are
symmetric in the variables $x_{m+1},x_{m+2},\dots$  was  studied in  \cite{L}. 
The $m$-symmetric Macdonald polynomials, which are indexed by $m$-partitions, are defined (up to a constant) in a finite number $N$ of variables as\footnote{The $m$-symmetric Macdonald polynomials have also been considered in \cite{BG}, where they are called ``partially-symmetric Macdonald polynomials''.}
\begin{equation*} 
  P_\Lambda(x;q,t) \propto \mathcal S^t_{m+1,\dots, N} E_\eta(x;q,t)
\end{equation*}
where $S^t_{m+1,\dots, N}$ is the $t$-symmetrization operator acting on the variables $x_{m+1},\dots,x_N$, and where $E_\eta(x;q,t)$ is a non-symmetric Macdonald polynomial (indexed by a composition that rearranges to $\Lambda$). The $m$-symmetric Macdonald polynomials contain as a subfamily the non-symmetric Macdonald polynomials indexed by compositions whose length is not larger than $m$, which
implies in particular that any non-symmetric Macdonald polynomial is an $m$-symmetric Macdonald polynomial for a sufficiently large $m$.

In \cite{L}, the goal was not so much to study the $m$-symmetric Macdonald polynomials per se but to study an extension to the $m$-symmetric world of the original Macdonald positivity conjecture (now theorem \cite{Haiman}):
\begin{equation} \label{Macdopos}
  \tilde J_\Lambda(x;q,t) = \sum_\Omega K_{\Omega \Lambda}(q,t) \, s_\Omega(x;t)
 \qquad {\rm with~} K_{\Omega \Lambda}(q,t) \in \mathbb N[q,t]??
\end{equation}
where $ \tilde J_\Lambda(x;q,t)$ is a plethystically modified version of the integral form of $P_\Lambda(x;q,t)$, and where $s_\Omega(x;t)$ are the $m$-symmetric Schur functions (a good portion of \cite{L} was dedicated to the study of those Schur functions now depending on a parameter $t$).
In this article, we focus instead on the properties of the $m$-symmetric Macdonald polynomials and show that they naturally extend those of the usual Macdonald polynomials (which correspond to the $m=0$ case). Most importantly,
just as is the case for Macdonald polynomials, there is a fundamental scalar product in $R_m$ with respect to which the  power sum symmetric functions in the $m$-symmetric world are orthogonal: 
\begin{equation*} 
\langle p_\Lambda(x;t)\, ,\, p_\Omega(x;t)  \rangle_m =\delta_{\Lambda \Omega} \, q^{|\pmb a|}t^{{\rm Inv} (\pmb a)} z_\lambda(q,t)
\end{equation*}
This product allows to define the $m$-symmetric Macdonald polynomials directly without having first to consider the case when the number of variables is finite. Indeed, 
the $m$-symmetric Macdonald polynomials form the 
  unique basis $\{ P_\Lambda(x;q,t)\}_\Lambda$ of $R_m$ such that (Proposition~\ref{propdefmacdo}):
  \begin{enumerate}
\item $\displaystyle{\langle P_\Lambda(x;q,t)\, ,\, P_\Omega(x;q,t)  \rangle_m = 0 {\rm ~if~} \Lambda\neq\Omega}  $ \quad {\rm (orthogonality)}    
\item $\displaystyle{P_\Lambda(x;q,t)= m_\Lambda + \sum_{\Omega < \Lambda} d_{\Lambda \Omega}(q,t) \, m_\Omega}$ \quad {\rm (unitriangularity)}
  \end{enumerate}
where the $m_\Lambda$'s are the $m$-symmetric monomials and where the order on $m$-partitions generalizes the usual dominance ordering.

Since $R_0 \subseteq R_1 \subseteq R_2 \subseteq \cdots$, it is natural to consider the inclusion $i: R_m \to R_{m+1}$ and the restriction $r:R_{m+1} \to R_m$ (which essentially amounts to letting $x_{m+1}=0$).  In Theorem~\ref{theoinclusion}, we obtain a simple formula for the inclusion 
$i(P_\Lambda)$ of an $m$-Macdonald polynomial (the restriction of an $m$-Macdonald polynomial was considered in \cite{L}).  Because of the relation
$$
\langle i(f), g \rangle_{m+1} =\langle f, r(g) \rangle_m  
$$
which holds for any $f \in R_m, g \in R_{m+1}$,  the inclusion and the restriction can then be used to
derive an explicit formula for the squared norm of the $m$-symmetric Macdonald polynomials in terms of arm and leg-lengths (Theorem~\ref{theoortho}):
$$
 \langle P_\Lambda(x;q,t)\, ,\, P_\Lambda(x;q,t)  \rangle_m =  q^{|\pmb a|}t^{{\rm Inv} (\pmb a)} \prod_{s \in \Lambda} \frac{1 -q^{\tilde a(s)+1}t^{\tilde \ell(s)}}{1 -q^{a(s)}t^{\ell(s)+1}}
$$
We obtain moreover a reproducing kernel for the scalar product (Corollary~\ref{coroKm}):
\begin{equation*} 
 {{t^{-{m\choose 2} }}} K_0(x,y)  T^{(x)}_{{\mathbf \omega}_m} \left[ \frac{\prod_{i+j \leq m} (1-tq^{-1} x_i y_j)}{\prod_{i+j \leq m+1}(1-q^{-1} x_i y_j)} \right]= \sum_{\Lambda} b_\Lambda(q,t) P_\Lambda(x;q,t) P_\Lambda(y;q,t)
  \end{equation*}
where the superscript $(x)$ indicates that the Hecke algebra operator  $T^{(x)}_{{\mathbf \omega}_m}$ acts on the $x$ variables, where  $K_0(x,y)$ is the Macdonald polynomial reproducing kernel, and where  $b_\Lambda(q,t)^{-1}=\| P_\Lambda(x;q,t)\|^2$.  Having established that the $m$-symmetric Macdonald polynomials satisfy the relation (Proposition~\ref{propoinv})
$$
q^{|\pmb a| } t^{{\rm Inv}(\pmb a)} P_\Lambda(x_m q^{-1},\dots,x_1 q^{-1},x_{m+1},x_{m+2},\dots;q^{-1},t^{-1}) = t^{{m \choose 2}}   \bar T_{\omega_m} P_\Lambda(x;q,t) 
$$
the aforementioned reproducing kernel can then be rewritten in a more elegant form as
$$
K_0(x,\tilde y)  \left[ \prod_{1\leq i <j \leq m}  \frac{1-tx_i y_j}{1-x_i y_j} \right]  \left[ \prod_{1\leq i \leq m}  \frac{1}{1-x_i y_i} \right] 
 =  \sum_{\Lambda}  a_\Lambda(q,t) P_\Lambda(x;q,t) P_\Lambda(y;q^{-1},t^{-1})
$$
where $\tilde y$ stands for the alphabet $\tilde y=(qy_1,\dots,q y_m,y_{m+1},y_{m+2},\cdots)$, and where the expansion coefficients are such that $a_\Lambda(q,t)=  q^{|\pmb a|}t^{{\rm Inv} (\pmb a)}  b_\Lambda(q,t)$.
In Remark~\ref{remarkcauchy}, we show that when the number of variables $N$ is finite and $N=m$, this Cauchy-type identity essentially reduces to the reproducing kernel for the non-symmetric Macdonald polynomials introduced in  \cite{MN}. This Cauchy-type identity also suggests that the scalar product $\langle \cdot,\cdot \rangle_m$ can be replaced by a sesquilinear scalar product as is discussed in Remark~\ref{remscalar}.

Finally, we derive an explicit formula for the principal specialization of an $m$-symmetric Macdonald polynomial (Proposition~\ref{propeval}):
\begin{equation*}
    P_{\Lambda}(1,t,\dots,t^{N-1};q,t) = \displaystyle t^{{n}(\Lambda)-{\rm coInv}(\pmb a)} \dfrac{[N-m]_t!}{[N]_t!} \prod_{s \in \Lambda^\circ} \dfrac{(1-q^{{a}'(s)}t^{N-\ell'(s)})}{(1-q^{{a}(s)}t^{\ell(s)+1})}
\end{equation*}
We then show that, as in the usual Macdonald polynomial case, there exists a natural extension $u_\Lambda$ of the principal specialization which satisfies the following symmetry (Proposition~\ref{SimMac}):
$$
u_\Omega(\tilde P_{\Lambda})=u_\Lambda (\tilde  P_{\Omega})
$$
where $\tilde P_{\Lambda}(x;q,t)=P_{\Lambda}(x;q,t)/P_{\Lambda}(1,t,\dots,t^{N-1};q,t)$.

The outline of the article is the following. The relevant definitions concerning the double affine Hecke algebra and the non-symmetric Macdonald polynomials are provided in Section~\ref{sec2}.  The extension to the $m$-symmetric world of the basic concepts in symmetric function theory are presented in Section~\ref{sec3}. The $m$-symmetric Macdonald polynomials are then defined in Section~\ref{sec4} along with a few of their elementary properties established in \cite{L}.  In Section~\ref{sec5}, formulas for the inclusion and the evaluation of an $m$-Macdonald polynomials are given.  In Section~\ref{sec6}, we introduce the scalar product with respect to which the $m$-symmetric Macdonald polynomials are later shown to be orthogonal.  We also obtain in Section~\ref{sec6} a reproducing kernel for the $m$-symmetric Macdonald polynomials (which is then shown to specialize to the reproducing kernel for the non-symmetric Macdonald polynomials), and the orthogonality/unitriangularity characterization of the $m$-symmetric Macdonald polynomials. Finally, the article ends with Appendix~\ref{appA}, \ref{appB} and \ref{appC} in which we have relegated a few technical proofs.

\section{Double affine Hecke algebra and non-symmetric Macdonald polynomials} \label{sec2}
The non-symmetric Macdonald polynomials can be defined as the common eigenfunctions of the Cherednik operators \cite{Che}, which are operators that belong to the double affine Hecke algebra and act on the ring $\mathbb Q(q,t)[x_1,\dots,x_N]$.  We now give the relevant definitions \cite{Mac2,Mar}.  
Let the exchange operator $K_{i,j}$ be such that
$$K_{i,j} f(\dots, x_i,\dots,x_j,\dots)= f(\dots, x_j,\dots,x_i,\dots)$$
We then define the generators $T_i$ of the affine Hecke algebra as
\begin{equation} \label{eqTi}
T_i=t+\frac{tx_i-x_{i+1}}{x_i-x_{i+1}}(K_{i,i+1}-1),\quad i=1,\ldots,N-1,
\end{equation}
and
$$
 {T_0=t+\frac{qtx_N-x_1}{qx_N-x_1}(K_{1,N}\tau_1\tau_N^{-1}-1)}\, ,
$$
where $\tau_i f(x_1,\dots, x_i,\dots,x_N)= f(x_1,\dots, qx_i,\dots,x_N)$ is the $q$-shift operator.
The $T_i$'s satisfy the relations  {($0\leq i\leq N-1$)}:
\begin{align*} &(T_i-t)(T_i+1)=0\nonumber\\
&T_iT_{i+1}T_i=T_{i+1}T_iT_{i+1}\nonumber\\
&T_iT_j=T_jT_i \, ,\quad i-j \neq \pm 1 \mod N
\end{align*}
where the indices are taken modulo $N$.
To define the Cherednik operators, we also need to introduce
the operator $\omega$ defined as:  
$$
\omega=K_{N-1,N}\cdots K_{1,2} \, \tau_1.
$$
We note that $\omega T_i=T_{i-1}\omega$ for $i=2,\dots,N-1$.

We are now in position to define the Cherednik operators:
$$
Y_i=t^{-N+i}T_i\cdots T_{N-1}\omega \bar T_1 \cdots \bar T_{i-1},
$$
where 
$$
\bar T_j :=  T_j^{-1}=t^{-1}-1+t^{-1}T_j,
$$
 which follows from the quadratic relation satisfied by the generators
 of the  Hecke algebra.
 The Cherednik operators obey the following  relations:
 \begin{align} \label{tsym1}
T_i \, Y_i&= Y_{i+1}T_i+(t-1)Y_i\nonumber \\
T_i \, Y_{i+1}&= Y_{i}T_i-(t-1)Y_i\nonumber \\
T_i Y_j & = Y_j T_i \quad {\rm if~} j\neq i,i+1.
\end{align}
It can be easily deduced from these relations that
\begin{equation}\label{TYi}
(Y_i+Y_{i+1})T_i= T_i (Y_i+Y_{i+1}) \qquad {\rm and } \qquad (Y_i Y_{i+1}) T_i =
T_i (Y_i Y_{i+1}). 
\end{equation}

An element  $\eta=(\eta_1,\dots,\eta_N)$ of $\mathbb Z_{\geq 0}^{N}$ is called a (weak) composition with $N$ 
parts (or entries).
It will prove convenient to represent a composition by a Young (or Ferrers) diagram.  Given a composition $\eta$ with $N$ parts, let $\eta^+$ be the partition obtained by reordering the entries of $\eta$.  The diagram corresponding to $\eta$ is the Young diagram of $\eta^+$ with an $i$-circle (a circle filled with an $i$) added to the right of the row of size $\eta_i$ (if there are many rows of size $\eta_i$, the circles are ordered from top to bottom in increasing order).  For instance, given $\eta=(0,2,1,3,2,0,2,0,0)$, we have
 $$
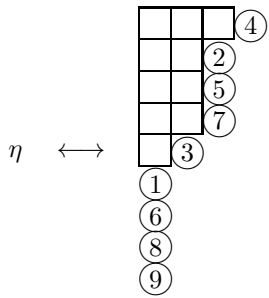
\eta \quad \longleftrightarrow  \quad {\tableau[scY]{&& & \bl \cercle{4}\\& & \bl \cercle{2} \\& & \bl \cercle{5}\\ & & \bl \cercle{7} \\ & \bl \cercle{3} \\ \bl \cercle{1} \\ \bl \cercle{6} \\ \bl \cercle{8} \\ \bl \cercle{9}}}
 $$

The Cherednik operators $Y_i$'s commute among each others, $[Y_i,Y_j]=0$,
and can be simultaneously diagonalized. Their eigenfunctions are the
 (monic) non-symmetric Macdonald polynomials (labeled by compositions).
For $x=(x_1,\dots,x_N)$, 
the non-symmetric Macdonald polynomial $E_\eta(x;q,t)$ is 
the 
unique polynomial with rational coefficients in $q$ and $t$ 
that is triangularly related to the monomials (in the Bruhat ordering on compositions) 
\begin{equation} \label{orderE}
E_\eta(x;q,t)=x^\eta+\sum_{\nu\prec\eta}b_{\eta\nu}(q,t) \, x^\nu
\end{equation}
and that satisfies, for all $i=1,\dots,N$, 
\begin{equation} 
  Y_i E_\eta=\bar \eta_iE_\eta,\qquad\text{where}\qquad  \bar\eta_i =q^{\eta_i}t^{1-r_\eta(i)} \label{eigenvalY}
\end{equation}
  with $r_\eta(i)$ standing for the row (starting from the top) in which the $i$-circle appears in
  the diagram of $\eta$.
 The Bruhat order on compositions is defined as follows:
 $$
   \nu\prec\eta\quad \text{ {iff}}\quad \nu^+<\eta^+\quad \text{or} \quad \nu^+=\eta^+\quad \text{and}\quad w_\eta < w_\nu,
 $$
 where $w_{\eta}$ is the unique permutation of minimal length such 
that $\eta = w_{\eta} \eta^+$ ($w_{\eta}$ permutes the entries of $\eta^+$).
In the Bruhat order on the symmetric group, $w_\eta {<} w_\nu$ iff
$w_{\eta}$ can be obtained as a  {proper} subword  of $w_{\nu}$.  The Cherednik operators have a triangular action on monomials \cite{Mac1}, that is,
\begin{equation} \label{triangY}
Y_i x^\eta =   \bar \eta_i x^\eta + {\rm ~smaller~terms}
\end{equation}
where ``smaller terms'' means that the remaining monomials $x^\nu$  appearing in the expansion are such that $\nu \prec \eta$.
  
The following three properties of the non-symmetric Macdonald polynomials are fundamental.
The first one expresses the stability of the polynomials $E_\eta$ with respect to the number of variables  (see e.g. \cite[eq. (3.2)]{Mar}):
\begin{equation} \label{property1}
E_\eta (x_1,\dots,x_{N-1},0;q,t) =
\left \{ 
\begin{array}{ll}
E_{\eta_-} (x_1,\dots,x_{N-1};q,t)
& {\rm if~} 
\eta_N = 0\, , \\
0 & {\rm if~} \eta_N \neq 0\, .
\end{array} \right.
 \end{equation}
where $\eta_-=(\eta_1,\ldots, \eta_{N-1}) $.   
The second one gives the action of the operators $T_i$ on $E_\eta$:
\begin{equation} \label{property2}
T_i E_{\eta} = \left\{ 
\begin{array}{ll}
\left(\frac{t-1}{1-\delta_{i,\eta}^{-1}}\right) E_\eta + t E_{s_i \eta} & {\rm if~} 
\eta_i < \eta_{i+1} \, ,  \\
t E_{\eta} &  {\rm if~} 
\eta_i = \eta_{i+1} \, ,\\
\left(\frac{t-1}{1-\delta_{i,\eta}^{-1}}\right) E_\eta + \frac{(1-t{\delta_{i,\eta}})(1-t^{-1}\delta_{i,\eta})}{(1-{\delta_{i,\eta}})^2} E_{s_i \eta} & {\rm if~} 
\eta_i > \eta_{i+1} \, ,
\end{array} \right. 
\end{equation}
where $\delta_{i,\eta}=\bar \eta_i/\bar \eta_{i+1}$ {and} $s_i \eta=(\eta_1,\dots,\eta_{i-1},\eta_{i+1},\eta_i,\eta_{i+2},\dots,\eta_N)$.
The third property, together with the previous one,  allows to construct the non-symmetric Macdonald polynomials recursively.  Given 
$\Phi_q=t^{1-N}T_{N-1}\cdots T_1 x_1$, we have that \cite{BF}
\begin{equation} \label{eqPhi}
 \Phi_q E_\eta(x;q,t) = t^{r_\eta(1)-N}  E_{\Phi \eta}(x;q,t)
\end{equation}
where $\Phi \eta=(\eta_2,\eta_3,\dots,\eta_{N-1},\eta_1+1)$.

Finally, we introduce the $t$-symmetrization operator:
\begin{equation} \label{symop}
\mathcal S^t_{m+1,N}=\sum_{\sigma\in S_{N-m}}T_\sigma
\end{equation}
where the sum is over the permutations in the symmetric group $S_{N-m}$
and
where
$$
    T_\sigma=T_{i_1+m}\cdots T_{i_\ell+m}
    \qquad  \text{if}\quad \sigma=s_{i_1}\cdots s_{i_\ell}.
$$
    We stress that there is a shift by $m$ in the indices of the Hecke algebra generators. Essentially, the $t$-symmetrization operator $\mathcal S^t_{m+1,N}$ acts on the variables $x_{m+1},x_{m+2},\dots, x_N$ while the usual $t$-symmetrization operator $\mathcal S^t_{1,N}$ acts on all variables $x_{1},x_{2},\dots, x_N$.
    \begin{remark}  \label{remarksym}
      We have by  \eqref{eqTi} that if a polynomial $f(x_1,\dots,x_N)$ is such that  $T_i f(x_1,\dots,x_N)=t  f(x_1,\dots,x_N)$, then $f(x_1,\dots,x_N)$ is symmetric in the variables $x_i$ and $x_{i+1}$.  As such, from \eqref{property2}, $E_\eta(x_1,\dots,x_N)$ is symmetric in the variables $x_i$ and $x_{i+1}$ whenever $\eta_i=\eta_{i+1}$.  We also have that $ \mathcal S^t_{m+1,N} f(x_1,\dots,x_N)$ is symmetric in the variables  $x_{m+1},x_{m+2},\dots, x_N$ for any polynomial $f(x_1,\dots,x_N)$ given that it can easily be checked that
\begin{equation}   \label{TionSymt}
  T_i \, \mathcal S^t_{m+1,N} =  \mathcal S^t_{m+1,N} \, T_i
  = t \, \mathcal S^t_{m+1,N} \qquad {\rm for~}i=m+1,\dots,N-1
\end{equation}
\end{remark}
 We should note that, up to a constant, the usual Macdonald polynomial, $P_\lambda(x_1,\dots,x_N;q,t)$, is obtained by $t$-symmetrizing a non-symmetric Macdonald polynomial:
$$
P_\lambda(x;q,t) \propto \mathcal S_{1,N}^{t} \, E_\eta(x;q,t)
$$
where $\eta$ is any composition that rearranges to $\lambda$.

Finally, the following relations will prove to be useful:
\begin{equation} \label{eqLR}
\mathcal S_{m+1,N}^t= \mathcal S_{m+2,N}^t \mathcal R_{m+1,N} = \mathcal L_{m+1,N} \mathcal S_{m+2,N}^t  = \mathcal L_{m+1,N}' \mathcal S_{m+1,N-1}^t 
\end{equation}
where
$$
\mathcal R_{m+1,N}=1+T_{m+1} + T_{m+1} T_{m+2} + \dots + T_{m+1} T_{m+2} \cdots T_{N-1}
$$
$$
\mathcal L_{m+1,N}=1+T_{m+1} + T_{m+2} T_{m+1} + \dots + T_{N-1} T_{N-2} \cdots T_{m+1}
$$
and
$$
\mathcal L_{m+1,N}'=1+T_{N-1} + T_{N-2} T_{N-1} + \dots + T_{m+1} T_{m+2} \cdots T_{N-1}
$$

\section{The ring of $m$-symmetric functions} \label{sec3}
Most of this section is taken from \cite{L}.
Let $\mathbf \Lambda=\mathbb Q(q,t)[h_1,h_2,h_3,\dots]$ be the ring of symmetric functions in the variables $x_1,x_2,x_3,\dots$ (the standard references on symmetric functions are \cite{M,Stan}), where
$$
h_r=h_r(x_1,x_2,x_3,\dots) = \sum_{i_1 \leq i_2 \leq \cdots \leq i_r} x_{i_1} x_{i_2} \cdots x_{i_r}
$$
 Bases of $\mathbf \Lambda$ are 
indexed by partitions $\lambda=(\lambda_1\geq\dots\geq\lambda_k>0)$ 
whose degree $\lambda$ is $|\lambda|=\lambda_1 +\cdots +\lambda_k$
and whose length $\ell(\lambda)=k$.  Each partition $\lambda$ has an 
associated Young diagram with $\lambda_i$ lattice squares in the $i^{th}$ 
row, from top to  bottom (English notation).  Any lattice square $(i,j)$ 
in the $i$th row and $j$th column of a Young diagram is called a cell.
 The partition
$\lambda\cup\mu$ is the non-decreasing rearrangement of the parts 
 of $\lambda$ and $\mu$.
  The dominance order $\geq$ is defined on partitions by
$\lambda\geq \mu$ when $\lambda_1+\cdots+\lambda_i\geq
\mu_1+\cdots+\mu_i$ for all $i$, and $|\lambda|=|\mu|$.

We define
the ring $R_m$ of $m$-symmetric functions as the subring of $\mathbb Q(q,t)[[x_1,x_2,x_3,\dots]]$ 
made  of formal power series that are symmetric  in the variables $x_{m+1},x_{m+2},x_{m+3},\dots$.
In other words, we have
$$R_m \simeq \mathbb Q(q,t)[x_1,\dots,x_m] \otimes \mathbf \Lambda_m$$
where
$\mathbf \Lambda_m$ is the ring of symmetric functions in the variables $x_{m+1},x_{m+2},x_{m+3},\dots$.
It is  immediate that $R_0=\mathbf \Lambda$ is the usual ring of symmetric functions and that $R_0 \subseteq R_1 \subseteq R_2 \subseteq \cdots $.  Bases of $R_m$ are naturally indexed by $m$-partitions which are pairs $\Lambda=(\pmb a;\lambda)$, where 
$\pmb a= (a_1,\dots,a_m) \in \mathbb Z_{\geq 0}^m$ is a composition with $m$ parts, and where
$\lambda$ is a partition.  We will call the entries of $\pmb a$ and $\lambda$ the
non-symmetric and symmetric entries of $\Lambda$ respectively.
In the following, unless stated otherwise, $\Lambda$ and $\Omega$ will always stand respectively for the $m$-partitions $\Lambda=(\pmb a;\lambda)$ and $\Omega=(\pmb b;\mu)$. Observe that we use a different notation for  the composition $\pmb a$ with $m$ parts (which corresponds to the non-symmetric entries 
of $\Lambda$)
than for the composition $\eta$ with $N$ parts (which will typically index a non-symmetric Macdonald polynomial).

Given a composition $\pmb a$ and a partition $\lambda$, $\pmb a \cup \lambda$ will denote the partition obtained by reordering the entries of the concatenation of $\pmb a$ and $\lambda$.
The degree of an $m$-partition $\Lambda$, denoted $|\Lambda|$, is the sum of the degrees of $\pmb a$ and $\lambda$, that is, 
$|\Lambda|=a_1+\dots +a_m+\lambda_1+\lambda_2+\cdots$. We also define the 
length of $\Lambda$ as $\ell(\Lambda)=m+\ell(\lambda)$. We will say that $\pmb a$ is dominant if $a_1 \geq a_2 \geq \cdots \geq a_m$, and by extension, we will say that $\Lambda=(\pmb a; \lambda)$ is dominant if $\pmb a$ is dominant.  If $\pmb a$ is not dominant, we let $\pmb a^+$ be the dominant composition obtained by reordering the entries of $\pmb a$.

There is a natural way to represent an $m$-partition by a Young diagram. 
The diagram corresponding to $\Lambda$ is the Young diagram of $\pmb a \cup \lambda$ with an $i$-circle  added to the right of the row of size $a_i$ for $i=1,\dots,m$ (if there are many rows of size $a_i$, the circles are ordered from top to bottom in increasing order).  For instance, given $\Lambda=(2,0,2,1; 3,2 )$, we have
 $$
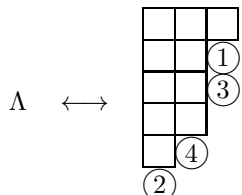
\Lambda \quad \longleftrightarrow  \quad {\tableau[scY]{&& & \bl \\& & \bl \cercle{1} \\& & \bl \cercle{3}\\ & & \bl  \\ & \bl \cercle{4} \\ \bl \cercle{2} }}
$$
Observe that when $m=0$, the diagram associated to  $\Lambda=( ; \lambda)$ coincides with the Young diagram associated to $\lambda$.
Also note that if $\eta$ is a composition with $m$ parts, then the diagram of $\eta$ coincides with the diagram of the $m$-partition $\Lambda=(\pmb a;\emptyset)$, where $\pmb a=\eta$.  We let $\Lambda^{(0)}=\pmb a \cup \lambda$, that is, $\Lambda^{(0)}$ is the partition 
obtained from the diagram of $\Lambda$ by discarding all the circles.
More generally, for $i=1,\dots,m$, we let $\Lambda^{(i)}=(\pmb a+1^i) \cup \lambda$, where $\pmb a +1^i=(a_1+1,\dots,a_i+1,a_{i+1},\dots,a_m)$. In other words, 
$\Lambda^{(i)}$ is the partition obtained from the diagram associated to $\Lambda$ by changing all of the $j$-circles, for $1 \leq j \leq i$, into squares and discarding the remaining circles.  Taking as above
$\Lambda=(2,0,2,1; 3,2)$, we have
$\Lambda^{(0)}=(3,2,2,2,1)$, $\Lambda^{(1)}=(3,3,2,2,1)$, $\Lambda^{(2)}=(3,3,2,2,1,1)$, $\Lambda^{(3)}=(3,3,3,2,1,1)$ and $\Lambda^{(4)}=(3,3,3,2,2,1)$. 
We then define the
dominance ordering on $m$-partitions to be  such that
\begin{equation} \label{deforder}
\Lambda \geq \Omega \iff \Lambda^{(i)} \geq \Omega^{(i)} \qquad {\rm for~all~}i=0,\dots,m
\end{equation}
where the order on the r.h.s.
is the usual dominance order on partitions.

We will associate arm and leg-lengths to the cells of the diagram of an $m$-partition.  Because of the circles, we will need two notions of arm-lengths as well as two notions of leg-lengths.  The arm-length  $a(s)$ 
is equal to the number of cells in $\Lambda$ strictly to the right of $s$ (and in the same row).  Note that if there is a circle at the end of its row, then it adds one to the arm-length of $s$.
 The arm-length $\tilde a(s)$ is exactly as $a(s)$ except that the circle at the end of the row does not contribute to $\tilde a(s)$.

The leg-length $\ell(s)$ is equal to the number of cells in $\Lambda$ strictly below $s$ (and in the same column). If at the bottom of its column there are $k$ circles whose fillings are smaller than the filling of the circle at the end of its row, then they add $k$ to the value of the leg-length of $s$.  If the row does not end with a circle then none of the circles at the bottom of its column  contributes to the leg-length.
 The leg-length $\tilde \ell(s)$ is exactly as $\ell(s)$ except that the circles at the bottom of the column contribute to $\tilde \ell(s)$ when there is no circle at the end of the row of $s$.

\begin{example} \label{exal}
The values of $a(s)$ and $\ell(s)$ in each cell of the diagram of  $\Lambda={(2,0,0,2;4,1,1)}$ are
$$
 {\tableau[scY]{\mbox{\scriptsize{\rm 34}}&\mbox{\scriptsize{\rm 22}}&  \mbox{\scriptsize{\rm 10}} & \mbox{\scriptsize{\rm 00}} & \bl \\\mbox{\scriptsize{\rm 23}}& \mbox{\scriptsize{\rm 11}}& \bl \cercle{\rm 1} \\\mbox{\scriptsize{\rm 24}}& \mbox{\scriptsize{\rm 10}} & \bl \cercle{\rm 4}\\ \mbox{\scriptsize{\rm 01}}& \bl  \\ \mbox{\scriptsize{\rm 00}} \\ \bl \cercle{\rm 2} \\ \bl \cercle{\rm 3}}}
 $$
while those of 
 $\tilde a(s)$ and $\tilde \ell(s)$ are
 $$
 {\tableau[scY]{\mbox{\scriptsize{\rm 36}}&\mbox{\scriptsize{\rm 22}}&  \mbox{\scriptsize{\rm 12}} & \mbox{\scriptsize{\rm 00}} & \bl \\\mbox{\scriptsize{\rm 13}}& \mbox{\scriptsize{\rm 01}}& \bl \cercle{\rm 1} \\\mbox{\scriptsize{\rm 14}}& \mbox{\scriptsize{\rm 00}} & \bl \cercle{\rm 4}\\ \mbox{\scriptsize{\rm 03}}& \bl  \\ \mbox{\scriptsize{\rm 02}} \\ \bl \cercle{\rm 2} \\ \bl \cercle{\rm 3}}}
 $$
\end{example}

Let the $m$-symmetric monomial function $m_\Lambda(x)$ be defined as
$$
m_\Lambda(x) := x_1^{a_1} \cdots x_m^{a_m} \, m_\lambda (x_{m+1},x_{m+2},\dots)= x^{\pmb a}  \, m_\lambda (x_{m+1},x_{m+2},\dots)
$$
where $m_\lambda (x_{m+1},x_{m+2},\dots)$ is the usual monomial symmetric function
 in the variables $x_{m+1},x_{m+2},\dots$
$$
m_\lambda (x_{m+1},x_{m+2},\dots)= \sum_\alpha x_{m+1}^{\alpha_1} x_{m+2}^{\alpha_2} \cdots   
$$
with the sum being over all  derrangements $\alpha$ of $(\lambda_1,\lambda_2,\dots,\lambda_{\ell(\lambda)},0,0,\dots)$.
It is immediate that $\{ m_\Lambda(x) \}_\Lambda$ is a basis of $R_m$.

Another basis of $R_m$ is provided by the $m$-symmetric power sums.
$$
p_\Lambda(x) := x_1^{a_1} \dots x_m^{a_m} \, p_\lambda (x)= x^{\pmb a} \, p_\lambda (x)
$$
 It should be observed that
 the variables in $p_\lambda$, contrary to those of $m_\lambda$ in $m_\Lambda(x)$, start at $x_1$ instead of
$x_{m+1}$. In this expression,  $p_\lambda (x)$ is the usual power-sum symmetric function
$$
p_\lambda(x) = \prod_{i=1}^{\ell(\lambda)} \, p_{\lambda_i}(x)
$$
where $p_r(x)=x_1^r+x_2^r+\cdots$.

Let $H_{\pmb a}(x_1,\dots,x_m;t)= E_{(a_1,\dots,a_m,0^{N-m})}(x_1,\dots,x_N;0,t)$ be the  non-symmetric Hall-Littlewood polynomial (the non-symmetric Macdonald polynomials only depend on the variables $x_1,\dots,x_m$ when $q=0$ and the indexing compositions have length at most $m$). For simplicity, we will denote 
the  non-symmetric Hall-Littlewood polynomial  $H_{\pmb a}(x;t)$ instead of  $H_{\pmb a}(x_1,\dots,x_m;t)$.
We should note that the polynomial  $H_{\pmb a}(x;t)$ can be constructed recursively as follows.  If 
$\pmb a$ is dominant then $H_{\pmb a}(x;t)=x^{\pmb a}$.  Otherwise,
$T_i H_{\pmb a}(x;t)=H_{s_i \pmb a}(x;t)$ if $a_i > a_{i+1}$ (with $s_i \pmb a=(a_1\dots,a_{i+1},a_i,\dots,a_m)$).  Since $H_{\pmb a}(x;1)=x^{\pmb a}$,
the following $t$-deformation of the $m$-symmetric power sum basis
\begin{equation} \label{eqplambda}
p_\Lambda(x;t)= H_{\pmb a}(x;t) p_\lambda(x)
\end{equation}
also provides a basis of $R_m$.

\section{$m$-symmetric Macdonald polynomials} \label{sec4}
The results of this section are again taken from \cite{L}.
The $m$-symmetric Macdonald polynomials in  $N$ variables are obtained by applying the $t$-symmetrization 
operator $\mathcal S_{m+1,N}^t$ introduced in \eqref{symop}
on non-symmetric Macdonald polynomials.
To be more specific, we associate to the $m$-partition $\Lambda= (\pmb a;\lambda)$ the composition $\eta_{\Lambda,N}={(a_1,\dots,a_m,\lambda_{N-m},...,\lambda_1)}$, where we consider that $\lambda_i=0$ if $i > \ell(\lambda)$.
The corresponding $m$-symmetric Macdonald polynomial in $N$ variables  is then defined as
\begin{equation}
P_\Lambda(x_1,\dots,x_N;q,t)= \frac{1}{u_{\Lambda,N}(t)} \, \mathcal S^t_{m+1,N } \, E_{\eta_{\Lambda,N}}(x_1,\dots,x_N;q,t)  
\end{equation}
with the normalization constant $u_{\Lambda,N}(t)$  given by
\begin{equation} \label{normalization}
u_{\Lambda,N}(t)=   \left( \prod_{i \geq 0} [n_\lambda(i)]_{t^{-1}}! \right)
t^{(N-m)(N-m-1)/2} 
\end{equation}
where $n_\lambda(i)$ is the number of entries in $\lambda_1,\dots,\lambda_{N-m}$ that are equal to $i$ (note that $i$ can be equal to zero), and where
$$
[k]_q=\frac{(1-q)(1-q^2)\cdots (1-q^k)}{(1-q)^k}
$$
Observe that the normalization constant $u_{\Lambda,N}(t)$ is chosen such that the coefficient of $m_\Lambda$ in $P_\Lambda(x;q,t)$ is equal to 1.
\begin{remark} \label{remarkgamma}
If
$\gamma$ is any composition such that $\gamma_i=a_i$ for $i=1,\dots,m$ and such that the remaining entries rearrange to $\lambda$, then
\begin{equation*}
P_\Lambda(x_1,\dots,x_N;q,t) =d_{\gamma}(q,t) \, \mathcal S^t_{m+1,N } \, E_{\gamma}(x_1,\dots,x_N;q,t)  
\end{equation*}
for some non-zero coefficient $d_{\gamma}(q,t) \in \mathbb Q(q,t)$.
This is an easy consequence of \eqref{property1} and \eqref{TionSymt} (see for
instance Lemma~55 in \cite{L} for a more precise statement).
\end{remark}

In the case where $\Lambda=(\pmb a;\emptyset)$, an $m$-symmetric Macdonald polynomial is simply a non-symmetric Macdonald polynomial:
\begin{equation} \label{eqPE}
P_{(\pmb a;\emptyset)} (x;q,t)= E_\eta(x,q,t)
\end{equation}  
where $\eta=(a_1,\dots,a_m,0^{N-m})$.

The $m$-symmetric Macdonald polynomials are stable with respect to the number of variables   
\begin{proposition} \label{propostable}
  Let $N$ be the number of variables and suppose that $N>m$.
Then
  $$
  P_\Lambda(x_1,\dots,x_{N-1},0;q,t)=
  \left \{ 
  \begin{array}{ll}
    P_{\Lambda}(x_1,\dots,x_{N-1};q,t) & {\rm if~} N>m+\ell(\lambda) \\
    0 & {\rm otherwise} 
  \end{array} \right .  
$$
\end{proposition}  

The $m$-symmetric Macdonald polynomials are the common eigenfunctions of a set of $m+1$ commuting operators.  First, they are eigenfunctions of the Cherednik operators $Y_i$, for  $i=1,\dots,m$:
\begin{equation} \label{defeigeni}
  Y_i P_\Lambda(x_1,\dots,x_N;q,t) = \varepsilon_\Lambda^{(i)}(q,t)  P_\Lambda(x_1,\dots,x_N;q,t) \qquad {\rm with} \qquad  \varepsilon_\Lambda^{(i)}(q,t)=  q^{a_i} t^{1-r_\Lambda(i)} 
\end{equation}
where we recall that $r_\Lambda(i)$ is the row in which the
$i$-circle  appears in the diagram associated to $\Lambda$.
They are also eigenfunctions of the operator
$$
D =  Y_{m+1}+\dots+Y_N - \sum_{i=m+1}^{N} t^{1-i}
$$
which is such that
\begin{equation} \label{defeigenD}
D \,  P_\Lambda(x_1,\dots,x_N;q,t)= \varepsilon_\Lambda^{D}(q,t)
P_\Lambda(x_1,\dots,x_N;q,t) \quad {\rm with} \quad  
 \varepsilon_\Lambda^{D}(q,t)=
 {\sum_{i}}' q^{\Lambda^{(0)}_i} t^{1-i} -\sum_{i=m+1}^{m+\ell (\lambda)} t^{1-i}
\end{equation}
 where the prime indicates that the sum is only over the rows of the diagram of $\Lambda$ that do not end with a circle.  We stress that the eigenvalues
 $ \varepsilon_\Lambda^{(i)}(q,t)$ and $ \varepsilon_\Lambda^{D}(q,t)$  do not depend on the number $N$ of variables and uniquely determine the $m$-partition $\Lambda$.

Letting the number of variables to be infinite, the $m$-symmetric Macdonald polynomials then form a basis of $R_m$.
 \begin{proposition} \label{propounitriang}
We have that
\begin{equation} \label{unitriang}
P_\Lambda(x;q,t)= m_\Lambda + \sum_{\Omega < \Lambda} d_{\Lambda \Omega}(q,t) \, m_\Omega
\end{equation}
Hence, the $m$-symmetric Macdonald polynomials  form a basis of $R_m$.
\end{proposition}

Finally, let
\begin{equation} \label{clambda}
c_\Lambda(q,t) = \prod_{s \in \Lambda} {(1-q^{{a}(s)}t^{\ell(s)+1})}
\end{equation}
where the product is over all the squares in the diagram of $\Lambda$ (not including the circles), and where the arm and leg-lengths were defined in Section~\ref{sec3}. The integral form of the $m$-symmetric Macdonald polynomials is then defined as
$J_\Lambda(x;q,t)=c_\Lambda(q,t) P_{\Lambda}(x;q,t)$.  It is this version that appears (after a certain plethystic substitution) in the positivity conjecture \eqref{Macdopos}.

\section{Inclusion, evaluation and symmetry} \label{sec5}

Since an $m$-symmetric function is also an $(m+1)$-symmetric function, it is natural to consider the inclusion $i: R_m \to R_{m+1}, a \mapsto a$.  The inclusion of  an $m$-symmetric Macdonald polynomial turns out to have a simple formula which will prove fundamental in the next section.  The proof, being quite long and technical, will be relegated to Appendix~\ref{appA}.
\begin{theorem} \label{theoinclusion}
The inclusion $i: R_m \to R_{m+1}$ is such that
$$
i(P_\Lambda) = \sum_{\Omega} \psi_{\Omega/\Lambda}(q,t) P_\Omega
$$
where the sum is over all $(m+1)$-partitions $\Omega$ whose diagram is  obtained from that of $\Lambda$ by adding an $(m+1)$-circle at the end of a symmetric row (a row that does not end with a circle).
The coefficient  $\psi_{\Omega/\Lambda}(q,t)$ is given explicitly as
$$
\psi_{\Omega/\Lambda}(q,t)=
\prod_{s \in {\rm col}_{\Omega/\Lambda}} \frac{ 1-q^{a_\Lambda(s)+1} t^{\tilde \ell_\Lambda (s)}}{ 1-q^{a_\Omega(s)+1} t^{\tilde \ell_\Omega (s)}}
$$
where ${\rm col_{\Omega/\Lambda}}$ stands for the set of squares in the diagram of $\Omega$ that lie in the column of the $(m+1)$-circle  and in a symmetric row (a row that does not end with a circle), and where  the arm and leg-lengths were defined before Example~\ref{exal} (with the indices specifying with respect to which $m$-partition they are computed).
\end{theorem}

Now, let ${\rm Inv}(\pmb a)$ be the number of inversions in $\pmb a$:
$$
{\rm Inv}(\pmb a)= \# \{ 1  \leq i <j \leq m \, | \, a_i < a_j \}
$$
and let
$$
{\rm coInv}(\pmb a)=m(m-1)/2-{\rm Inv}(\pmb a)=\# \{ 1  \leq i <j \leq m \, | \, a_i \geq a_j \}
$$
As usual, for a partition $\lambda$, we let  $n(\lambda)=\sum_i (i-1)\lambda_i$.
In the case of an $m$-partition, we will define $n(\Lambda):=n(\Lambda^{(m)})$,
where we recall that $\Lambda^{(m)}$ is the partition obtained from the diagram of $\Lambda$ by converting all the circles into squares.

We now give the principal specialization $u_{\emptyset}\bigl(f(x_1,x_2,\dots,x_N)\bigr):=f(1,t,\dots,t^{N-1})$ of an $m$-symmetric Macdonald polynomial.

\begin{proposition} \label{propeval}
  For $\Lambda=(\pmb a; \lambda)$, the principal specialization is given by
\begin{equation*}
    u_{\emptyset}(P_{\Lambda}(x;q,t)) = \displaystyle t^{{n}(\Lambda)-{\rm coInv}(\pmb a)} \dfrac{[N-m]_t!}{[N]_t!} \prod_{s \in \Lambda^\circ} \dfrac{(1-q^{{a}'(s)}t^{N-\ell'(s)})}{(1-q^{{a}(s)}t^{\ell(s)+1})}
\end{equation*}
where $\Lambda^\circ$ stands for the set of cells (including the circles)
in the diagram of $\Lambda$, and where the co-arm and co-leg are given respectively by $a'(s)=j-1$ and $\ell'(s)=i-1$ for the cell $s=(i,j)$.
\end{proposition}

\begin{proof}
Define the operator
\begin{equation*}
    \Psi_N =(1-t)(1+T_{N-1}+T_{N-2}T_{N-1}+\cdots + T_{m} \cdots T_{N-1}) \Phi_q
\end{equation*}
where $\Phi_q$ was introduced in \eqref{eqPhi}.  The action of $\Psi_N$ on an $m$-Macdonald polynomial turns out to be quite simple \cite{L}:
\begin{equation} \label{eqPsiJ}
    \Psi_N J_{\Lambda}(x;q,t) = t^{-\#\{2 \leq j \leq m \, |\, a_j \leq a_1 \}}J_{\Lambda^{\Box}}(x,q,t)
\end{equation}
where $\Lambda^\Box=\bigl(a_2,\dots,a_m;\lambda \cup{(a_1+1)}\bigr)$,
and where the integral form of the $m$-Macdonald polynomials was introduced in Section~\ref{sec4}. Note that
the diagram of $\Lambda^\Box$ can be obtained from that of $\Lambda$ by 
transforming the $1$-circle into a square (and then relabeling the 
remaining circles so that they go from 1 to $m-1$ instead of from 2 to $m$).

  Using $(T_if)(1,\dots,t^{N-1})=t f(1,\dots,t^{N-1})$ for all $i=1,\dots,N-1$, and for all $f(x_1,\dots,x_N) \in \mathbb Q[x_1,\dots,x_N]$, we easily deduce that
$$
u_\emptyset(\Psi_N g)=(1-t)(1+t+\cdots+t^{N-m}) u_\emptyset (g)=(1-t^{N-m+1}) u_\emptyset (g)
$$
for all $g(x_1,\dots,x_N) \in \mathbb Q[x_1,\dots,x_N]$. Applying $u_{\emptyset}$ 
on both sides of \eqref{eqPsiJ} and using   $J_\Lambda(x;q,t)=c_\Lambda(q,t) P_\Lambda(x;q,t)$, we thus get that
\begin{equation} \label{recursi}
    u_{\emptyset}(P_{\Lambda}) = \dfrac{t^{-\#\{2 \leq j \leq m \, |\, a_j \leq a_1 \}}}{(1-t^{N-m+1}) } \dfrac{c_{\Lambda^{\Box}}(q,t)}{c_{\Lambda}(q,t)} u_{\emptyset}(P_{\Lambda^{\Box}})
\end{equation}
Since $\Lambda^\Box$ is an $(m-1)$-partition while $\Lambda$ is an $m$-partition, this recursion will allow us to prove the proposition by induction on $m$.

In the base case $m=0$, we have 
${\rm coInv}(\pmb a)=0$ and
$\Lambda=(;\lambda)=\Lambda^{(m)}=\Lambda^\circ$ can be identified with $\lambda$.  Hence the proposition simply becomes
\begin{equation*}
    u_{\emptyset}(P_{\lambda}(x;q,t)) = \displaystyle t^{{n}(\lambda)}  \prod_{s \in \lambda} \dfrac{(1-q^{{a}'(s)}t^{N-\ell'(s)})}{(1-q^{{a}(s)}t^{\ell(s)+1})}
\end{equation*}
which is the well-known evaluation of a Macdonald polynomial \cite{M}.

We will now see that the general case holds.  Let
$$
c'_\Lambda(q,t) = \prod_{s \in \Lambda^\circ} (1-q^{ a(s)}t^{\ell(s)+1})  \qquad {\rm and} \qquad
d_\Lambda(q,t) = \prod_{s \in \Lambda^\circ} (1-q^{a'(s)}t^{N-\ell'(s)}) 
$$
First observe that
$$
\frac{c_{\Lambda}(q,t)}{c'_{\Lambda}(q,t)} = \frac{c_{\Lambda^\Box}(q,t)}{(1-t)c'_{\Lambda^\Box}(q,t)} 
$$
since the 1-circle contributes in $c'_\Lambda(q,t)/c_\Lambda(q,t)$ while not in $c'_{\Lambda^\Box}(q,t)/c_{\Lambda^\Box}(q,t)$. Using $d_\Lambda(q,t)=d_{\Lambda^\Box}(q,t)$  and \eqref{recursi}, the 
 proposition will thus hold by induction if we can show that
\begin{equation*}
t^{{n}(\Lambda)-{\rm coInv}(\pmb a)} \dfrac{[N-m]_t!}{[N]_t!} \frac{1}{(1-t)}   = \dfrac{t^{-\#\{2 \leq j \leq m \, |\, a_j \leq a_1 \}}}{(1-t^{N-m+1}) }  t^{{n}(\Lambda^\Box)-{\rm coInv}(\pmb a')} \dfrac{[N-m+1]_t!}{[N]_t!}
\end{equation*}
where $\pmb a'=(a_2,\dots,a_m)$.  But this easily follows from
$n(\Lambda)=n(\Lambda^\Box)$,
$$
{\rm coInv}(\pmb a) = {\rm coInv}(\pmb a')+\#\{2 \leq j \leq m \, |\, a_j \leq a_1 \}
$$
and
$$
[N-m+1]_t! = \frac{(1-t^{N-m+1})}{(1-t)} [N-m]_t!
$$
\end{proof}
In  the special case $m=N$, our evaluation formula for the  $m$-symmetric Macdonald polynomials can be simplified.  It provides a reformulation of the principal specialization of the non-symmetric Macdonald polynomials which can be found for instance in \cite{Mar}.
\begin{corollary} \label{eval Non} The non-symmetric Macdonald polynomials are such that
\begin{equation*}
    E_{\eta}(1,t,\dots,t^{N-1};q,t) = \displaystyle t^{n(\eta^+)+{\rm Inv(\eta)}} \prod_{s \in \eta} \dfrac{(1-q^{{a}(s)}t^{N-\ell'(s)})}{(1-q^{{a}(s)}t^{\ell(s)+1})}
\end{equation*}
where we recall that $\eta^+$ is the partition obtained by reordering the entries of $\eta$.
\end{corollary} 
\begin{proof}
From \eqref{eqPE}, we have in the case $m=N$ that
$P_{(\pmb a;\emptyset)}(x;q,t)=E_{\eta}(x;q,t)$ with $\eta=(a_1,\dots,a_N)$. 
Using Proposition~\ref{propeval}, we thus get that
\begin{equation} \label{eqspecial}
    E_{\eta}(1,t,\dots,t^{N-1};q,t) = \displaystyle t^{n(\eta^+)+N(N-1)/2-{\rm coInv(\eta)}} \frac{1}{[N]_t!} \prod_{s \in \eta^\circ} \dfrac{(1-q^{{a}'(s)}t^{N-\ell'(s)})}{(1-q^{{a}(s)}t^{\ell(s)+1})}
\end{equation}
where we have used the fact that $n(\Lambda)=n(\eta^++1^N)=n(\eta^+)+N(N-1)/2$ since every row of $\eta$ ends with a circle.  It is straightforward to check that
$$
\prod_{s \in \eta^\circ} (1-q^{{a}'(s)}t^{N-\ell'(s)})= \prod_{s \in \eta^\circ} (1-q^{{a}(s)}t^{N-\ell'(s)})   
$$
Hence 
\begin{equation} \label{eqcircle}
 \prod_{s \in \eta^\circ} \dfrac{(1-q^{{a}'(s)}t^{N-\ell'(s)})}{(1-q^{{a}(s)}t^{\ell(s)+1})} =    \left[ \prod_{s \in \eta} \dfrac{(1-q^{{a}(s)}t^{N-\ell'(s)})}{(1-q^{{a}(s)}t^{\ell(s)+1})} \right]  \left[ \prod_{s \in \circ} \dfrac{(1-q^{{a}(s)}t^{N-\ell'(s)})}{(1-q^{{a}(s)}t^{\ell(s)+1})} \right] 
\end{equation}
where $\circ$ stands for the cells of the diagram of $\eta$ corresponding to circles.  Note that when $s$ is in the position of a circle, we have $a(s)=0$ and $\ell(s)=0$.  Because there is a circle in every row of $\eta$, we thus obtain
 $$
 \prod_{s \in \circ} \dfrac{(1-q^{{a}(s)}t^{N-\ell'(s)})}{(1-q^{{a}(s)}t^{\ell(s)+1})}=[N]_t
 $$
Using the previous result in \eqref{eqcircle}, the corollary follows immediately from \eqref{eqspecial} and the relation ${\rm Inv}(\eta)=N(N-1)/2-{\rm coInv}(\eta)$.
\end{proof}  

We now introduce an evaluation depending on an $m$-partition.  We will see that it satisfies a natural symmetry property.  First, to the $m$-partition $\Lambda=(a_1,\dots,a_m;\lambda)$, we associate the composition
$$
\gamma_\Lambda=(a_1,\dots,a_m,\lambda_1,\dots,\lambda_\ell,0^{N-m-\ell})
$$
Let $w$ be the minimal length permutation such that  $w\gamma_\Lambda$ is weakly decreasing.  If we forget about the extra zeroes, we thus have that
$w\gamma_\Lambda= \Lambda^{(0)}$,
where we recall that $\Lambda^{(0)}$ stands for the partition whose diagram is obtained from that of $\Lambda$ by removing all the circles.
 The evaluation $u_{\Lambda}$  is then defined on any $m$-symmetric function $f(x)$ as
\begin{equation} \label{eval}
    u_{\Lambda}\bigl(f(x_1,\dots,x_N)\bigr)=  f\left(q^{-\Lambda_{w(1)}^{(0)}}t^{w(1)-1},\dots,q^{-\Lambda_{w(N)}^{(0)}}t^{w(N)-1}\right)
\end{equation}
Observe that in the case $\Lambda = (0^m;\emptyset)$, $u_\Lambda$ corresponds to the principal evaluation $u_\emptyset$.

\begin{lemma}\label{AutVal} If $f$ is an $m$-symmetric function in $N$ variables then
\begin{equation} \label{aut1}
    f(Y^{-1})  P_{\Lambda}(x;q,t) = u_\Lambda(f)   P_{\Lambda}(x;q,t)
\end{equation}
\end{lemma}

\begin{proof}  Let $\eta=\gamma_{\Lambda}$.
We have that $Y_i^{-1}E_{\eta} = \bar \eta_i^{-1} E_{\eta}$,
where we recall that $\bar \eta_i= q^{\eta_i} t^{1-r_\eta(i)}$. Using the fact that $f$ is $m$-symmetric, we then have from Remark~\ref{remarkgamma} that
\begin{equation*}
\begin{array}{ll}
  f(Y_1^{-1},\dots,Y_N^{-1})   P_\Lambda(x;q,t) &= \displaystyle f(Y_1^{-1},\dots,Y_N^{-1})  d_\eta(q,t)  \mathcal S^t_{m+1,,N}  E_{\eta}  \\
&= \displaystyle   d_\eta(q,t)   \mathcal S^t_{m+1,,N} f(Y_1^{-1},\dots,Y_N^{-1}) E_{\eta}  \\
   &= \displaystyle   d_\eta(q,t)    \mathcal S^t_{m+1,,N} f(\bar \eta_1^{-1},\dots, \bar \eta_N^{-1}) E_{\eta}  \\
     &=   f(\bar \eta_1^{-1},\dots, \bar \eta_N^{-1})    P_\Lambda(x;q,t)
\end{array}
\end{equation*}
It thus only remains to show that the specialization
$x_i = \bar{\eta}^{-1}_i$ corresponds to the evaluation defined in \eqref{eval}.
Let $w$ be the minimal length permutation such that $w\eta=\Lambda^{(0)}$.
We have immediately that $\eta_i=\Lambda^{(0)}_{w(i)}$.  Hence, 
from the definition of $u_\Lambda$, we only need to show that
$r_\eta(i)=w(i)$.  But this is a consequence of the minimality of $w$. Indeed,
if $\eta_i=\eta_j$ and $i<j$ then the minimality of $w$ ensures that $w(i)<w(j)$, which implies that 
the circles increase from top to bottom in equal rows.
\end{proof}
The next proposition extends a well-known property of the Macdonald polynomials \cite{M}. Recall that $u_{\emptyset}( P_{\Lambda}(x,q,t))$ was given explicitly in Proposition~\ref{propeval}.
\begin{proposition}\label{SimMac} Let $\tilde  P_{\Lambda}(x,q,t)$ be the normalization of the 
 $m$-Macdonald polynomials given by
  $$
\tilde P_{\Lambda}(x,q,t) = \dfrac{ P_{\Lambda}(x;q,t)}{u_{\emptyset}( P_{\Lambda}(x,q,t))}
$$
Then, the following symmetry holds:
$$
u_\Omega(\tilde P_{\Lambda})=u_\Lambda (\tilde  P_{\Omega})
$$
\end{proposition} 
\begin{proof}  For $f(x)$ and $g(x)$ Laurent polynomials in $x_1,\dots,x_N$, it is known \cite{Mac2} that the pairing
  $$
[f(x),g(x)] := u_{\emptyset} \left( f(Y^{-1}) g(x)  \right)
  $$
is such that $[f,g]=[g,f]$.  From Lemma~\ref{AutVal}, we thus get
\begin{align*}
    \left[{ P}_{\Lambda}(x_{(N)},q,t), { P}_{\Omega}(x_{(N)},q,t) \right] &= u_{\emptyset}\bigl({ P}_{\Lambda}(Y^{-1}_i) { P}_{\Omega}(x_{(N)},q,t)\bigr) \\
    &= u_\Omega({ P}_{\Lambda}(x_{(N)},q,t)) u_{\emptyset}\bigl( { P}_{\Omega}(x_{(N)},q,t)\bigr)
\end{align*}
From the symmetry of the pairing $[\cdot,\cdot]$, it then follows that
$$
u_\Omega({ P}_{\Lambda}(x_{(N)},q,t)) u_{\emptyset}\bigl( { P}_{\Omega}(x_{(N)},q,t)\bigr)=
 u_\Lambda({ P}_{\Omega}(x_{(N)},q,t)) u_{\emptyset}\bigl({ P}_{\Lambda}(x_{(N)},q,t)\bigr)
 $$
which proves the proposition. 
\end{proof}

The final result of this section is concerned with the behavior of $P_\Lambda(x;q,t)$ when $q$ and $t$ are sent to $q^{-1}$ and $t^{-1}$.   For $\sigma \in S_N$ with a reduced decomposition $s_{i_1} \cdots s_{i_r}$, we let $K_\sigma= K_{i_1,i_1+1} \cdots K_{i_r,i_r+1}$, and $T_\sigma=T_{i_1} \cdots T_{i_r}$. 
We also let $\omega_m=[m,m-1,\dots,1]$ be the longest permutation in the symmetric group $S_m$ (which we consider as the element $[m,\dots,1,m+1,\dots,N]$ of $S_N$), and denote the inverse of $T_{\omega_m}$ by $\bar T_{\omega_m}$.
\begin{proposition} \label{propoinv} We have that
$$
q^{|\pmb a| } t^{{\rm Inv}(\pmb a)} P_\Lambda(x;q^{-1},t^{-1}) = t^{{m \choose 2}} \tau_1 \cdots \tau_m K_{\omega_m}  \bar T_{\omega_m} P_\Lambda(x;q,t)
$$
or, equivalently, that
$$
q^{|\pmb a| } t^{{\rm Inv}(\pmb a)} P_\Lambda(x_m q^{-1},\dots,x_1 q^{-1},x_{m+1},x_{m+2},\dots;q^{-1},t^{-1}) = t^{{m \choose 2}}   \bar T_{\omega_m} P_\Lambda(x;q,t) 
$$
\end{proposition}
\begin{remark}
The proposition is an extension of a similar result on non-symmetric polynomials (see Lemma 2.3 a) in \cite{Mar}) which states that
\begin{equation} \label{relNS}
t^{{\rm Inv} ( \eta)} E_\eta(x_m,\dots,x_1;q^{-1},t^{-1})= {{t^{{m\choose 2} }}}  \bar T_{{\mathbf \omega}_m} E_\eta(x_1,\dots,x_m;q,t)
\end{equation}
This relation was also proven in a broader context in  \cite{A} in connection with permuted-basement Macdonald polynomials  \cite{HHL}. When the number of variables $N$ is equal to $m$ and $\Lambda=(\eta;\emptyset)$, Proposition~\ref{propoinv} becomes \eqref{relNS} (the $q$ powers canceling from the homogeneity of $E_\eta$). But when the number of variables is larger than $m$ and $\Lambda=(\eta;\emptyset)$, Proposition~\ref{propoinv} is actually stronger than \eqref{relNS} since it says that for any non-symmetric Macdonald polynomial such that $\ell(\eta) \leq m$ we have 
$$
q^{|\eta| } t^{{\rm Inv}(\eta)} E_\eta(x_m q^{-1},\dots,x_1 q^{-1},x_{m+1},x_{m+2},\dots;q^{-1},t^{-1}) = t^{{m \choose 2}}   \bar T_{\omega_m} E_\eta (x;q,t) 
$$
\end{remark}
\begin{proof}[Proof of Proposition~\ref{propoinv}] It suffices to prove the result in $N$ variables.  
Let $Y_i^\star$ be the Cherednik operator $Y_i$ with parameters $q^{-1}$ and $t^{-1}$ instead of $q$ and $t$, and similarly for the operator $D^\star$. We thus have from \eqref{defeigeni} and \eqref{defeigenD} that
$$
Y_i^\star \, P_\Lambda(x;q^{-1},t^{-1})= \varepsilon_\Lambda^{(i)}(q^{-1},t^{-1}) \, P_\Lambda(x;q^{-1},t^{-1}) \qquad i=1,\dots,m 
$$
and
$$
D^\star \, P_\Lambda(x;q^{-1},t^{-1})= \varepsilon_\Lambda^{D}(q^{-1},t^{-1}) \, P_\Lambda(x;q^{-1},t^{-1})
$$
The main part of the proof thus consists in proving that $\tau_1 \cdots \tau_m K_{\omega_m}  \bar T_{\omega_m} P_\Lambda(x;q,t)$ is an eigenfunction of $Y_1^\star,\dots,Y_m^\star$ and $D^\star$ with the right eigenvalues. In order to achieve this, we prove that for any $f \in R_m$ we have
\begin{equation} \label{eqinv1}
Y_i^\star \left( \tau_1 \cdots\tau_m K_{\omega_m} \bar T_{\omega_m} \right) f= \left(\tau_1 \cdots \tau_m  K_{\omega_m} \bar T_{\omega_m} \right) \bar Y_i f \qquad i=1,\dots,m 
\end{equation}
and
\begin{equation} \label{eqinv2}
D^\star  \left( \tau_1 \cdots \tau_m K_{\omega_m} \bar T_{\omega_m} \right) f=  \left( \tau_1 \cdots \tau_m K_{\omega_m} \bar T_{\omega_m} \right) \bar D f 
\end{equation}
where $\bar D=\bar Y_{m+1}+\cdots +\bar Y_N + \sum_{i=m+1}^N t^{i-1}$. It is then immediate that 
\begin{equation*}
\begin{split}
Y_i^\star \left( \tau_1 \cdots \tau_m K_{\omega_m} \bar T_{\omega_m} \right)  P_\Lambda(x;q,t) & = \left(\tau_1 \cdots \tau_m  K_{\omega_m} \bar T_{\omega_m} \right) \bar Y_i  P_\Lambda(x;q,t)\\
& = \varepsilon_\Lambda^{(i)}(q^{-1},t^{-1})  \left(\tau_1 \cdots \tau_m  K_{\omega_m} \bar T_{\omega_m} \right)  P_\Lambda(x;q,t)  \qquad i=1,\dots,m 
\end{split}
\end{equation*}
and
\begin{equation*}
\begin{split}
D^\star  \left( \tau_1 \cdots \tau_m K_{\omega_m} \bar T_{\omega_m} \right)  P_\Lambda(x;q,t) & =  \left( \tau_1 \cdots \tau_m K_{\omega_m} \bar T_{\omega_m} \right) \bar D  P_\Lambda(x;q,t) \\
&=  \varepsilon_\Lambda^{D}(q^{-1},t^{-1})  \left( \tau_1 \cdots \tau_m K_{\omega_m} \bar T_{\omega_m} \right)   P_\Lambda(x;q,t) 
\end{split}
\end{equation*} 
as wanted. The proof of \eqref{eqinv1} and \eqref{eqinv2}, which is somewhat technical, is provided in Appendix~\ref{appC}.

We have thus proven that $\tau_1 \cdots \tau_m K_{\omega_m}  \bar T_{\omega_m} P_\Lambda(x;q,t)$ is equal to $P_\Lambda(x;q^{-1},t^{-1})$ up to a constant. Hence, we have left to prove that the proportionality constant corresponds to the powers of $q$ and $t$ in the statement of the proposition.
In the proof of Lemma 2.3 a) in \cite{Mar}, it is shown that
$$
 K_{\omega_m}  \bar T_{\omega_m} x^{\pmb a} = t^{{\rm Inv} ( \pmb a)-{{m\choose 2}}} x^{\pmb a} + {\rm smaller~terms}
$$
where the order is the order on compositions defined in Section~\ref{sec2}. We thus deduce straightforwardly that the action of $\tau_1 \cdots \tau_m K_{\omega_m}  \bar T_{\omega_m}$
on the dominant term $m_\Lambda=x^{\pmb a} m_\lambda(x_{m+1},\dots,x_N)$ of $P_\Lambda(x;q^{-1},t^{-1})$ is such that
$$
\tau_1 \cdots \tau_m K_{\omega_m}  \bar T_{\omega_m} m_\Lambda = q^{|\pmb a|} t^{{\rm Inv} ( \pmb a)-{{m\choose 2}}}m_\Lambda + {\rm smaller~terms}
$$   
where the smaller terms are of the form $x^{\pmb b} m_\lambda(x_{m+1},\dots,x_N)$ with $\pmb b$ smaller than $\pmb a$.
This concludes the proof of the proposition.
\end{proof}

\section{Orthogonality} \label{sec6}

Recall that ${\rm Inv}(\pmb a)$ is the number of inversions in $\pmb a$, and let
 $|\pmb a|=a_1+\cdots +a_m$. The following scalar product in $R_m$ is defined on the $t$-deformation of the $m$-symmetric power sums introduced in \eqref{eqplambda}:
\begin{equation} \label{eqscal}
\langle p_\Lambda(x;t)\, ,\, p_\Omega(x;t)  \rangle_m =\delta_{\Lambda \Omega} \, q^{|\pmb a|}t^{{\rm Inv} (\pmb a)} z_\lambda(q,t)
\end{equation}
where
$$
z_\lambda(q,t) = z_\lambda \prod_{i=1}^{\ell(\lambda)} \frac{1-q^{\lambda_i}}{1-t^{\lambda_i}}
$$
with $z_\lambda= \prod_{i \geq 1 } i^{n_\lambda(i)} \cdot n_\lambda(i)!$ (recall that $n_\lambda(i)$ is the number of occurrences of $i$ in $\lambda$). Observe that
when $m=0$, this corresponds to the usual Macdonald polynomial scalar product \cite{M}. 

The main goal of this section is to show that the $m$-symmetric Macdonald polynomials are orthogonal with respect to the scalar product \eqref{eqscal} and to provide the value of the squared norm $ \|P_\Lambda(x;q,t) \|^2$.
But before proving the theorem, we need to establish a few results.

The Macdonald polynomial reproducing kernel 
$$
K_0(x,y)= \prod_{i,j} \frac{(tx_i y_j;q)_{\infty}}{(x_i y_j;q)_{\infty}} \qquad {\rm with}\quad  (a;q)_\infty=\prod_{i=1}^{\infty} (1-aq^{i-1})
$$
is such that  \cite{M}
 \begin{equation} \label{eqkernel0}
K_0(x,y)=\sum_{\lambda}  z_{\lambda}(q,t)^{-1}  p_{\lambda} (x) p_{\lambda} (y)
\end{equation}
Recalling that $\omega_m$ was introduced a the end of the previous section, we define
\begin{equation} 
K_m(x,y)=  {{t^{-{m\choose 2} }}} K_0(x,y)  T^{(x)}_{{\mathbf \omega}_m} \left[ \frac{\prod_{i+j \leq m} (1-tq^{-1} x_i y_j)}{\prod_{i+j \leq m+1}(1-q^{-1} x_i y_j)} \right]
  \end{equation}
where the superscript $(x)$ indicates that the Hecke algebra operator  $T^{(x)}_{{\mathbf \omega}_m}$ acts on the $x$ variables.  We will see later in Proposition~\ref{propokernel} that
$K_m(x,y)$ is a reproducing kernel for the scalar product \eqref{eqscal}.

We first show that the eigenoperators $Y_1,\dots,Y_m,D$ are symmetric with respect to $K_m(x,y)$ when the number of variables is finite. In order not to disrupt the flow of the presentation, the proof will be relegated to Appendix~\ref{appB}.
\begin{proposition} \label{propsymfinite}  For $x_{(N)}=(x_1,\dots,x_N)$ and $y_{(N)}=(y_1,\dots,y_N)$, we have that
  $$
  Y_i^{(x)} K_m(x_{(N)},y_{(N)}) =  Y_i^{(y)} K_m(x_{(N)},y_{(N)}) 
  $$
  for $i=1,\dots,m$,
  and
  $$
D^{(x)} K_m(x_{(N)},y_{(N)}) =  D^{(y)} K_m(x_{(N)},y_{(N)})
  $$
\end{proposition}
As already mentioned, the eigenvalues of the $Y_i$'s and $D$ do not depend on the number of variables $N$. Using Proposition~\ref{propostable}, we can thus define  the operators $\tilde Y_i$ and $\tilde D$ as their inverse limits.  In other words, the operators $\tilde Y_i: R_m \to R_m$ (for $i=1,\dots,m$) and $\tilde D: R_m \to R_m$ are defined such that
$$
\tilde Y_i P_\Lambda(x;q,t)=\varepsilon^{(i)}_\Lambda(q,t) P_\Lambda(x;q,t)  \quad {\rm and} \quad \tilde D P_\Lambda(x;q,t)=\varepsilon^D_\Lambda(q,t) P_\Lambda(x;q,t)
$$
for all $m$-partitions $\Lambda$,
where $\varepsilon^{(i)}_\Lambda(q,t)$ and $\varepsilon^D_\Lambda(q,t)$ are such as defined in \eqref{defeigeni} and \eqref{defeigenD} respectively.  We will now see that the previous proposition also holds for $\tilde Y_i$ and $\tilde D$.
\begin{proposition} \label{propsym}  
  We have that
  $$
  \tilde Y_i^{(x)} K_m(x,y) =  \tilde Y_i^{(y)} K_m(x,y) 
  $$
  for $i=1,\dots,m$,
  and
  $$
\tilde D^{(x)} K_m(x,y) =  \tilde D^{(y)} K_m(x,y)
$$
Moreover,
 \begin{equation} \label{eqKm}
K_m(x,y) = \sum_{\Lambda} b_\Lambda(q,t) P_\Lambda(x;q,t) P_\Lambda(y;q,t)
  \end{equation}
for certain coefficients $b_\Lambda(q,t) \in \mathbb Q(q,t)$ that will be given explicitely in Corollary~\ref{coroKm}.
\end{proposition}
\begin{proof}
  From \eqref{eqkernel0} and the fact that  $K_m(x,y)/K_0(x,y)$ only depends on $x_1,\dots,x_m$, we have that
  $$
K_m(x,y)= \sum_{\Lambda,\Omega} d_{\Lambda \Omega}(q,t) P_\Lambda(x;q,t)  P_\Omega(y;q,t)
$$
for some coefficients $d_{\Lambda \Omega}(q,t)$. From Proposition~\ref{propostable},
we then get in $N$ variables that
 $$
K_m(x_{(N)},y_{(N)})= \sum_{\ell(\Lambda),\ell(\Omega) \leq N } d_{\Lambda \Omega}(q,t) P_\Lambda(x_{(N)};q,t)  P_\Omega(y_{(N)};q,t)
$$
Therefore, if for instance the action of $\tilde D^{(x)}$ were different from  that of $\tilde D^{(y)}$ then there would exist a coefficient $d_{\Lambda \Omega}(q,t)$ such that
$\varepsilon^D_\Lambda (q,t) d_{\Lambda \Omega}(q,t) \neq \varepsilon^D_\Omega (q,t) d_{\Lambda \Omega}(q,t)$.
But then, choosing $N$ large enough, this would contradict the fact that  $D^{(x)}$ and $D^{(y)}$
have the same action on $K_m(x_{(N)},y_{(N)})$. Hence, the first part of the proposition holds.

Now, this entails that for all $\Lambda,\Omega$ we have
$$
\varepsilon^D_\Lambda (q,t) d_{\Lambda \Omega}(q,t) = \varepsilon^D_\Omega (q,t) d_{\Lambda \Omega}(q,t) \qquad {\rm and} \qquad \varepsilon^{(i)}_\Lambda (q,t) d_{\Lambda \Omega}(q,t) = \varepsilon^{(i)}_\Omega (q,t) d_{\Lambda \Omega}(q,t), \quad i=1,\dots,m
$$
Given that the eigenvalues as a whole uniquely determine $\Lambda$, we must have that $d_{\Lambda \Omega}(q,t)=0$ if $\Lambda \neq \Omega$. Letting $b_\Lambda(q,t)=d_{\Lambda \Lambda}(q,t)$, we get that \eqref{eqKm} also holds.
\end{proof}

The following proposition will be instrumental in the proof that $K_m(x,y)$ is a reproducing kernel for the scalar product \eqref{eqscal}.
\begin{proposition} \label{propt0} We have
  $$
 {{t^{-{m\choose 2} }}} T^{(x)}_{{\mathbf \omega}_m} \left[ \frac{\prod_{i+j \leq m} (1-t x_i y_j)}{\prod_{i+j \leq m+1}(1- x_i y_j)} \right] = \sum_{\pmb a } t^{-{\rm Inv}(\pmb a)} H_{\pmb a}(x;t)  H_{\pmb a}(y;t) 
 $$
where the sum is over all $\pmb a \in \mathbb Z_{\geq 0}^m$.
 \end{proposition}
\begin{proof} Letting $y_i \mapsto q y_i$ in \eqref{eqKm}, we obtain
  \begin{align}
  {{t^{-{m\choose 2} }}}  K_0(x,qy) T^{(x)}_{{\mathbf \omega}_m} \left[ \frac{\prod_{i+j \leq m} (1-t x_i y_j)}{\prod_{i+j \leq m+1}(1- x_i y_j)}\right] =K_m(x,qy) & =\sum_{\Lambda} b_{\Lambda}(q,t) P_{\Lambda} (x;q,t) P_{\Lambda} (qy;q,t) \nonumber \\
  & =\sum_{\Lambda} b_{\Lambda}(q,t) q^{|\Lambda|} P_{\Lambda} (x;q,t) P_{\Lambda} (y;q,t)
  \nonumber
  \end{align}
by the homogeneity of $P_{\Lambda} (y;q,t)$. Using  $P_{\Lambda} (x_1,\dots,x_m;q,t)=0$
if $\ell(\Lambda) > m$ by Proposition~\ref{propostable}, we get when 
restricting to $m$ variables that
$$
 {{t^{-{m\choose 2} }}}  K_0(\bar x,q \bar y)  T^{(x)}_{{\mathbf \omega}_m} \left[ \frac{\prod_{i+j \leq m} (1-t x_i y_j)}{\prod_{i+j \leq m+1}(1- x_i y_j)} \right] =\sum_{\pmb a} \bar b_{(\pmb a ;\emptyset)}(q,t) P_{(\pmb a ;\emptyset)} (\bar x;q,t) P_{(\pmb a ;\emptyset)} (\bar y;q,t)
 $$
 where $\bar x=(x_1,\dots,x_m)$ (and similarly for $\bar y$), and where $\bar b_{(\pmb a ;\emptyset)}(q,t)= b_{(\pmb a ;\emptyset)}(q,t) q^{|\pmb a|}$.  Letting $q=0$ and using $ K_0(\bar x,\bar z)=1$ whenever $\bar z=(0,\dots,0)$,
 we then obtain
 \begin{equation} \label{eqdbar}
 {{t^{-{m\choose 2} }}}   T^{(x)}_{{\mathbf \omega}_m} \left[  \frac{\prod_{i+j \leq m} (1-t x_i y_j)}{\prod_{i+j \leq m+1}(1- x_i y_j)} \right]  =\sum_{\pmb a}  \bar b_{(\pmb a ;\emptyset)}(0,t) H_{\pmb a} (   x;t) H_{\pmb a} (   y;t)
 \end{equation}
Note that we have used the fact that 
$P_{(\pmb a ;\emptyset)} (\bar x;0,t)=  H_{\pmb a} (\bar x;t)= H_{\pmb a}( x;t)$ given that $H_{\pmb a} (   x;t)$ only depends on $x_1,\dots,x_m$.
 We thus only need to show that $ \bar b_{(\pmb a ;\emptyset)}(0,t)=t^{-{\rm Inv}(\pmb a)}$, which will be achieved by using the specialization
 $y_i= t^{i-1}$ in the previous equation.  First observe that setting $q=0$ in Corollary~\ref{eval Non} yields  $H_{\pmb a} (   1,t,\dots,t^{m-1};t)=t^{n(\pmb a ^+)} t^{{\rm Inv}(\pmb a)}$ ($a(s)$ is always larger than zero  given that every row
 in the diagram of $\eta=\pmb a$ ends with a circle), and that
 $$
 \frac{\prod_{i+j \leq m} (1- x_i t^j)}{\prod_{i+j \leq m+1}(1- x_i t^{j-1})} = \frac{1}{\prod_{i=1}^m(1- x_i)}
 $$
 Specializing the variables $y$ in \eqref{eqdbar}  thus gives
 $$
{{t^{-{m\choose 2} }}}  T^{(x)}_{{\mathbf \omega}_m} \left[ \frac{1}{\prod_{i=1}^m(1- x_i)} \right] =  \frac{1}{\prod_{i=1}^m(1- x_i)}  =\sum_{\pmb a}  \bar b_{(\pmb a ;\emptyset)}(0,t) t^{n(\pmb a ^+)} t^{{\rm Inv}(\pmb a)} H_{\pmb a} (   x;t) 
 $$
since $\prod_{i=1}^m(1- x_i)$ is symmetric.
  Finally, we observe that $\prod_{i=1}^m(1- tx_i)^{-1}$ is the generating series of the complete symmetric functions  $h_n(x_1,\dots,x_m)$  to deduce  \cite{M} that
  $$
 \frac{1}{\prod_{i=1}^m(1- x_i)} = \sum_n h_n(x_1,\dots,x_m)= \sum_{\lambda \, ; \, \ell(\lambda) \leq m}  t^{n(\lambda)} P_\lambda(\bar x;t) = \sum_{\lambda \, ; \, \ell(\lambda) \leq m} t^{n(\lambda)}  \sum_{\pmb a\, : \, \pmb a^+ = \lambda} H_{\pmb a}(x;t)
  $$
 where $ P_\lambda(\bar x;t)$ is a Hall-Littlewood polynomial, and where the elementary relation $P_\lambda(\bar x;t)=\sum_{\pmb a\, : \, \pmb a^+ = \lambda} H_{\pmb a}(x;t)$ can be found for instance in \cite{Mac1}. Comparing the previous two equations, we obtain immediately that $\bar b_{(\pmb a ;\emptyset)}(0,t) = t^{-{\rm Inv}(\pmb a)}$, as wanted.
 \end{proof}  
\begin{corollary} We have
  $$
K_m(x,y)=\sum_{\Lambda} q^{-|\pmb a|} t^{-{\rm Inv}(\pmb a)} z_{\lambda}(q,t)^{-1}  p_{\Lambda} (x;t) p_{\Lambda} (y;t)
  $$
\end{corollary}  
\begin{proof} 
 Letting $y_i \mapsto q^{-1} y_i$ in Proposition~\ref{propt0} yields
$$
 {{t^{-{m\choose 2} }}}  T^{(x)}_{{\mathbf \omega}_m} \left[ \frac{\prod_{i+j \leq m} (1-t q^{-1} x_i y_j)}{\prod_{i+j \leq m+1}(1- q^{-1 }x_i y_j)} \right]
 =\sum_{\pmb a} q^{-|\pmb a|}t^{-{\rm Inv}(\pmb a)} H_{\pmb a} (x;t) H_{\pmb a} (y;t)
 $$
 since  $H_{\pmb a} (q^{-1}y;t)=q^{-|\pmb a|} H_{\pmb a}(y;t)$. The corollary then follows from \eqref{eqkernel0} and the definition of $p_\Lambda(x;t)$.
\end{proof}
We immediately get that
 the function $K_m(x,y)$ is a reproducing kernel for the scalar product \eqref{eqscal}.
 \begin{proposition} \label{propokernel} Let $\{ f_\Lambda(x) \}_\Lambda$ and $\{ g_\Lambda(x) \}_\Lambda$ be two bases of $R_m$.  Then the following two statements are equivalent.
   \begin{enumerate}
   \item $\displaystyle{K_m(x,y) = \sum_\Lambda f_\Lambda(x) g_\Lambda(y)}$ 
    \item  $\displaystyle{\langle f_\Lambda(x) \, ,\, g_\Omega(x) \rangle_m =\delta_{\Lambda \Omega} \quad {\rm for~all~} \Lambda,\Omega}$.
        \end{enumerate}  
\end{proposition}
 \begin{proof}
   Using the bases $\{ p_\Lambda^*(x;t) \}_\Lambda$ and $\{ p_\Lambda(x;t) \}_\Lambda$, where $ p_\Lambda^*(x;t)= q^{|\pmb a|} t^{{\rm Inv}(\pmb a)} z_{\lambda}(q,t)   p_\Lambda(x;t)$, the proof is exactly as the proof of the similar statement in the usual Macdonald polynomial case \cite{M}.
 \end{proof}  

Before stating the main theorem of this section, we need to relate the inclusion and the restriction.  First,
it is straightforward to verify that the inclusion $i: R_m \to R_{m+1}$   
is such that \cite{L}
$$
i(p_\Lambda(x;t)) =p_{\Lambda^0}(x;t)
$$
where 
$\Lambda^0=(\pmb a,0;\lambda)$. The restriction $r: R_{m+1} \to R_m$, which is defined as
$$
r(f)=f(x_1,\dots,x_m,0,x_{m+2},x_{m+3},\dots)  \big |_{(x_{m+2},x_{m+3},\dots) \mapsto (x_{m+1},x_{m+2},\dots) }
$$
is on the other hand such that
\cite{L}
\begin{equation} \label{rH}
r(p_\Omega(x;t)) =
\left \{ \begin{array}{ll}
  p_{\Omega_-}(x;t) & {\rm if~} b_{m+1}=0 \\
  0 & {\rm otherwise}
\end{array} \right .  
\end{equation}
where 
$\Omega_-=(\pmb b_-;\mu)$ with $\pmb b_-=(b_1,\dots,b_m)$.

The following proposition can  easily be verified using the basis $\left\{ p_\Lambda(x;t)\right\}_\Lambda$ of $R_m$.
\begin{proposition} \label{propir} We have 
\begin{equation} \label{dualir}
\langle i(f),g \rangle_{m+1} = \langle f, r(g) \rangle_m
\end{equation}
for all $f\in R_m$ and all $g \in R_{m+1}$.
\end{proposition}
We can now establish the orthogonality and the squared norm of the $m$-symmetric Macdonald polynomials.
\begin{theorem} \label{theoortho}
We have
 $$
\langle P_\Lambda(x;q,t)\, ,\, P_\Omega(x;q,t)  \rangle_m = 0  \qquad {\rm if~} \Lambda\neq\Omega
$$
and 
$$
 \langle P_\Lambda(x;q,t)\, ,\, P_\Lambda(x;q,t)  \rangle_m =  q^{|\pmb a|}t^{{\rm Inv} (\pmb a)} \prod_{s \in \Lambda} \frac{1 -q^{\tilde a(s)+1}t^{\tilde \ell(s)}}{1 -q^{a(s)}t^{\ell(s)+1}}
$$
where the product is over the cells of $\Lambda$ (excluding the circles), and where  the arm and leg-lengths were defined before Example~\ref{exal}.
\end{theorem}

\begin{example} Using $\Lambda=(2,0,0,2;4,1,1)$ such as in Example~\ref{exal}, we get that $\|P_\Lambda(x;q,t) \|^2$ is given by
$$
 \fr{(1-q)(1-q^2t^2)(1-q^3t^2)(1-q^4t^6)(1-qt)(1-q^2t^3)(1-q)(1-q^2t^4)(1-qt^3)(1-qt^2)}{(1-t)(1-qt)(1-q^2t^3)(1-q^3t^5)(1-qt^2)(1-q^2t^4)(1-qt)(1-q^2t^5)(1-t^2)(1-t)} 
$$
\end{example}
\begin{proof}
Proposition~\ref{propokernel} and \eqref{eqKm} immediately imply that the $m$-symmetric Macdonald polynomials are orthogonal, that is,
     $$
\langle P_\Lambda(x;q,t)\, ,\, P_\Omega(x;q,t)  \rangle_m = 0  \qquad {\rm if~} \Lambda\neq\Omega
$$

We thus only have to prove the formula for the squared norm of an $m$-Macdonald polynomial.  Let $\Lambda = (a_1,\ldots,a_{m-1},a_{m};\lambda)$ and $\hat{\Lambda} = (a_1,\ldots a_{m-1}; \lambda \cup \{a_{m}\} )$. Observe that $\hat{\Lambda}$ can be obtained from $\Lambda$ by discarding the $m$-circle.  The restriction of an 
 $m$-Macdonald polynomial is given (in the integral form) by \cite{L}
$$
  r({J}_{\Lambda}(x,q,t)) = q^{a_{m}} t^{\#\{i \, |\, a_i < a_{m}\}}  {J}_{\hat \Lambda}(x,q,t)
$$
which amounts to
\begin{equation} \label{eqr}
    r({P}_{\Lambda}(x,q,t)) = q^{a_{m}} t^{\#\{i \, |\, a_i < a_{m}\}} \varphi_{\Lambda / \hat{\Lambda}}(q,t) {P}_{\hat \Lambda}(x,q,t)
\end{equation}
where
\begin{equation*} 
    \varphi_{\Lambda / \hat{\Lambda}}(q,t)=\frac{c_{\hat \Lambda}(q,t)}{c_{ \Lambda}(q,t)} = \displaystyle \prod_{s \in \Lambda} \dfrac{1-q^{a_{\hat{\Lambda}}(s)}t^{\ell_{\hat{\Lambda}}(s)+1} }{1-q^{a_{\Lambda}(s)}t^{\ell_{\Lambda}(s)+1}} 
\end{equation*}
From the definition of the arm and leg-length, we see that
$a_{\hat{\Lambda}}(s) = a_{\Lambda}(s)$ and $\ell_{\hat{\Lambda}}(s) = \ell_{\Lambda}(s)$
for all $s \in \Lambda$ except those in ${\rm row}_{\Lambda/\hat \Lambda}$, the row in which the $m$-circle of $\Lambda$ lies.  Hence
\begin{equation} \label{varphi}
    \varphi_{\Lambda / \hat{\Lambda}}(q,t) = \displaystyle \prod_{s \in  {\rm row}_{\Lambda/\hat \Lambda}} \dfrac{1-q^{a_{\hat{\Lambda}}(s)}t^{\ell_{\hat{\Lambda}}(s)+1} }{ 1-q^{a_{\Lambda}(s)}t^{\ell_{\Lambda}(s)+1}} 
\end{equation}
Also observe that the formula for the inclusion in Theorem~\ref{theoinclusion} gives
\begin{equation}\label{psi}
    \psi^{-1}_{\Lambda / \hat{\Lambda}}(q,t) =  \prod_{s \in  {\rm col}_{\Lambda / \hat{\Lambda}}} \dfrac{1-q^{\tilde{a}_{{\Lambda}}(s)+1}t^{\tilde{\ell}_{{\Lambda}}(s)} }{ 1-q^{\tilde{a}_{\hat \Lambda}(s)+1 }t^{\tilde{\ell}_{\hat \Lambda}(s)}  }
\end{equation}
where we observe that for  $s \in {\rm col}_{\Lambda / \hat{\Lambda}}$
we have that $\tilde{a}(s) = a(s)$.

The proof will proceed by induction on $m$.
  In the case $m=0$, as was already observed, the scalar product is the usual Macdonald polynomial scalar product.  Using
  $\Lambda = \lambda$, $|\pmb a| = 0$, ${\rm Inv}(\pmb{a}) = 0$, $\tilde{a}(s) = a(s)$ y $\tilde{\ell}(s) = \ell(s)$, we thus have to show that
  the Macdonald polynomials are such that
\begin{equation*}
    \langle {P}_{\lambda}(x,q,t) , {P}_{\lambda}(x,q,t) \rangle_0 =  \displaystyle \prod_{s \in \lambda} \dfrac{1-q^{a(s)+1}t^{\ell(s)}}{1-q^{a(s)}t^{\ell(s)+1}} 
\end{equation*}
But this is the well known formula for the norm squared of a Macdonald polynomial \cite{M}.

Supposing that the theorem holds for the $(m-1)$-symmetric Macdonald polynomials,  we will see that it also holds for the $m$-symmetric Macdonald polynomials.
Let $\Lambda$ and $\hat \Lambda$ be as before.  From the formula for the inclusion of an $(m-1)$-symmetric Macdonald polynomial, we have
\begin{equation} \label{1}
    i({ P}_{\hat{\Lambda}}(x,q,t)) = \sum_{\Omega} \psi_{\Omega / \hat{\Lambda}}(q,t) {P}_{\Omega}(x,q,t)
\end{equation}
where $\Omega$ is obtained from  $\hat{\Lambda}$ by adding an $m$-circle.
Taking $\Omega=\Lambda$, we get from the orthogonality of the $m$-Macdonald polynomials that
\begin{equation*}
    \langle i({ P}_{\hat{\Lambda}}(x,q,t)) , { P}_{\Lambda}(x,q,t) \rangle_{m} = \psi_{\Lambda / \hat{\Lambda}}(q,t) \langle {P}_{\Lambda}(x,q,t) , {P}_{\Lambda}(x,q,t) \rangle_{m}  
\end{equation*}
or equivalently, that 
\begin{equation*}
    \langle { P}_{\Lambda}(x,q,t) , { P}_{\Lambda}(x,q,t) \rangle_{m}  = \psi^{-1}_{\Lambda / \hat{\Lambda}} (q,t)    \langle i({P}_{\hat{\Lambda}}(x,q,t)) , {P}_{\Lambda}(x,q,t) \rangle_{m}
\end{equation*}
From Proposition~\ref{propir}, we then obtain
\begin{equation*}
    \langle { P}_{\Lambda}(x,q,t) , {P}_{\Lambda}(x,q,t) \rangle_{m}    =  \psi^{-1}_{\Lambda / \hat{\Lambda}}(q,t) \langle {P}_{\hat{\Lambda}}(x,q,t) , r({P}_{\Lambda}(x,q,t)) \rangle_{m-1} 
\end{equation*}
which amounts, using \eqref{eqr}, to
\begin{equation*}
    \langle {P}_{\Lambda}(x,q,t) , {P}_{\Lambda}(x,q,t) \rangle_{m}    =  q^{a_{m+1}}t^{\{i\, | \,a_i < a_{m+1}\}}\psi^{-1}_{\Lambda / \hat{\Lambda}}(q,t)  \varphi_{\Lambda / \hat{\Lambda}}(q,t) \langle { P}_{\hat{\Lambda}}(x,q,t) , {P}_{\hat{\Lambda}}(x,q,t) \rangle_{m-1} 
\end{equation*}
By induction, we thus get
\begin{align*}
    \langle { P}_{\Lambda}(x,q,t) , {P}_{\Lambda}(x,q,t) \rangle_{m} &= q^{a_{m}}t^{\{i\, | \, a_i < a_{m}\}}\psi^{-1}_{\Lambda / \hat{\Lambda}}(q,t)  \varphi_{\Lambda / \hat{\Lambda}}(q,t) q^{|\hat{\pmb {a}}|}t^{{\rm Inv}(\hat{\pmb {a}})} \displaystyle \prod_{s \in \hat \Lambda} \dfrac{1-q^{\tilde{a}_{\hat{\Lambda}}(s)+1}t^{\tilde{\ell}_{\hat{\Lambda}}(s)}}{1-q^{a_{\hat{\Lambda}}(s)}t^{\ell_{\hat{\Lambda}}(s)+1}} \\
     &= q^{|\pmb {a}|}t^{{\rm Inv}(\pmb {a})} \psi^{-1}_{\Lambda / \hat{\Lambda}}(q,t)  \varphi_{\Lambda / \hat{\Lambda}}(q,t) \displaystyle \prod_{s \in \Lambda} \dfrac{1-q^{\tilde{a}_{\hat{\Lambda}}(s)+1}t^{\tilde{\ell}_{\hat{\Lambda}}(s)}}{1-q^{a_{\hat{\Lambda}}(s)}t^{\ell_{\hat{\Lambda}}(s)+1}}
\end{align*}
Now, $\tilde{a}_{\Lambda}(s)= \tilde{a}_{\hat{\Lambda}}(s)$ for all $s \in \Lambda$,
$\tilde{\ell}_{\Lambda}(s)= \tilde{\ell}_{\hat{\Lambda}}(s)$ for all $s \in \Lambda\setminus {\rm col}_{\Lambda / \hat{\Lambda}}$ while  $a_{\hat{\Lambda}}(s) = a_{\Lambda}(s)$ and $\ell_{\hat{\Lambda}}(s) = \ell_{\Lambda}(s)$ for all  $s \in \Lambda\setminus {\rm row}_{\Lambda / \hat{\Lambda}}$.  The squared norm of $P_\Lambda$ is thus equal to
\begin{equation*}
     q^{|\pmb {a}|}t^{{\rm Inv}(\pmb {a})} \psi^{-1}_{\Lambda / \hat{\Lambda}}(q,t)  \varphi_{\Lambda / \hat{\Lambda}}(q,t)  \dfrac{\displaystyle \prod_{\Lambda\setminus {\rm col}_{\Lambda / \hat{\Lambda}}} (1-q^{\tilde{a}_{\Lambda}(s)+1}t^{\tilde{\ell}_{\Lambda}(s)})}{\displaystyle \prod_{\Lambda\setminus {\rm row}_{\Lambda / \hat{\Lambda}}} (1-q^{a_{\Lambda}(s)}t^{\ell_{\Lambda}(s)+1})} \dfrac{\displaystyle \prod_{s \in {\rm col}_{\Lambda / \hat{\Lambda}}} (1-q^{\tilde{a}_{\hat{\Lambda}}(s)+1}t^{\tilde{\ell}_{\hat{\Lambda}}(s)})}{\displaystyle \prod_{s \in  {\rm row}_{\Lambda / \hat{\Lambda}}} (1-q^{a_{\hat{\Lambda}}(s)}t^{\ell_{\hat{\Lambda}}(s)+1})}
\end{equation*}
Finally, using \eqref{varphi} and \eqref{psi}, we obtain
\begin{equation*}
      \langle { P}_{\Lambda}(x,q,t) , {P}_{\Lambda}(x,q,t) \rangle_{m} = q^{|\pmb {a}|}t^{{\rm Inv}(\pmb {a})} \dfrac{\displaystyle \prod_{s \in \Lambda} (1-q^{\tilde{a}_{\Lambda}(s)+1}t^{\tilde{\ell}_{\Lambda}(s)})}{\displaystyle \prod_{s \in \Lambda} (1-q^{a_{\Lambda}(s)}t^{\ell_{\Lambda}(s)+1})}
\end{equation*}
as wanted.
\end{proof}
\begin{corollary} \label{coroself} The operators $\tilde Y_i$, for $i=1,\dots,m$, and $\tilde D$ defined before Proposition~\ref{propsym} are self-adjoint with respect to the scalar product $\langle \cdot, \cdot \rangle_m$, that is,
$$
\langle \tilde Y_i f, g \rangle_m = \langle  f, \tilde Y_i g \rangle_m  \quad {\rm for~} i=1,\dots,m \quad {\rm and} \quad \langle \tilde D f, g \rangle_m = \langle  f, \tilde D g \rangle_m
$$
\end{corollary}

We can also immediately deduce from Proposition~\ref{propokernel} the value of the coefficients $b_\Lambda (q,t)$ in \eqref{eqKm}.
\begin{corollary}\label{coroKm} We have
 \begin{equation} \label{eqcoroKm}
K_m(x,y) = \sum_{\Lambda} b_\Lambda(q,t) P_\Lambda(x;q,t) P_\Lambda(y;q,t)
  \end{equation}
 where
 $$
b_\Lambda(q,t)^{-1} = \langle P_\Lambda(x;q,t)\, ,\, P_\Lambda(x;q,t)  \rangle_m =  q^{|\pmb a|}t^{{\rm Inv} (\pmb a)} \prod_{s \in \Lambda} \frac{1 -q^{\tilde a(s)+1}t^{\tilde \ell(s)}}{1 -q^{a(s)}t^{\ell(s)+1}}
 $$
\end{corollary}
The function $K_m(x,y)$ is not totally explicit due to the presence of the operator $T_{\omega_m}$. Using Proposition~\ref{propoinv}, this defect in Corollary~\ref{coroKm} can be corrected.
\begin{proposition} The following Cauchy-type identity holds
    \begin{equation}  \label{eqidentity}
K_0(x,\tilde y)  \left[ \prod_{1\leq i <j \leq m}  \frac{1-tx_i y_j}{1-x_i y_j} \right]  \left[ \prod_{1\leq i \leq m}  \frac{1}{1-x_i y_i} \right] 
 =  \sum_{\Lambda}  a_\Lambda(q,t) P_\Lambda(x;q,t) P_\Lambda(y;q^{-1},t^{-1})
\end{equation}
where $\tilde y$ stands for the alphabet
$$
\tilde y=(qy_1,\dots,q y_m,y_{m+1},y_{m+2},\cdots) 
$$
and where
 $$
a_\Lambda(q,t)^{-1} =  q^{-|\pmb a|}t^{-{\rm Inv} (\pmb a)} b_\Lambda(q,t)^{-1}= \prod_{s \in \Lambda} \frac{1 -q^{\tilde a(s)+1}t^{\tilde \ell(s)}}{1 -q^{a(s)}t^{\ell(s)+1}}
 $$
\end{proposition}
\begin{proof} We get from \eqref{eqcoroKm} that $K_m(x,y)$ is symmetric in $x$ and $y$.  Therefore, using $T_{\omega_m}^{(y)}$ instead of $T_{\omega_m}^{(x)}$ in $K_m(x,y)$, we obtain from \eqref{eqcoroKm} that
$$
K_0(x,y) \left[ \frac{\prod_{i+j \leq m} (1-tq^{-1} x_i y_j)}{\prod_{i+j \leq m+1}(1-q^{-1} x_i y_j)} \right] = \sum_{\Lambda} b_\Lambda(q,t) P_\Lambda(x;q,t) \left[ t^{m \choose 2} \bar T_{\omega_m}^{(y)} P_\Lambda(y;q,t) \right]
$$
Applying $(\tau_1 \cdots \tau_m K_{\omega_m})^{(y)}$ on both sides of the equation and using Proposition~\ref{propoinv}
thus yields 
$$
K_0(x,\tilde y) K_{\omega_m}^{(y)} \left[ \frac{\prod_{i+j \leq m} (1-t x_i y_j)}{\prod_{i+j \leq m+1}(1-x_i y_j)} \right] = \sum_{\Lambda} b_\Lambda(q,t) P_\Lambda(x;q,t) \left[ q^{|\pmb a|}t^{{\rm Inv} (\pmb a)}  P_\Lambda(y;q^{-1},t^{-1}) \right]
$$
The proposition is then immediate after checking that
$$
 K_{\omega_m}^{(y)} \left[ \frac{\prod_{i+j \leq m} (1-t x_i y_j)}{\prod_{i+j \leq m+1}(1-x_i y_j)} \right] =  \left[ \prod_{1\leq i <j \leq m}  \frac{1-tx_i y_j}{1-x_i y_j} \right]  \left[ \prod_{1\leq i \leq m}  \frac{1}{1-x_i y_i} \right] 
 $$   
\end{proof}

\begin{remark} \label{remscalar} The previous proposition suggests that there is a natural  
sesquilinear scalar product $\langle \cdot, \cdot \rangle'$ in $R_m$ such that
\begin{equation} \label{eqscalar}
\langle P_\Lambda(x;q,t), P_\Omega(x;q,t) \rangle' = \delta _{\Lambda \Omega}\, c_\Lambda(q,t)^{-1}=\delta _{\Lambda \Omega} \prod_{s \in \Lambda} \frac{1 -q^{\tilde a(s)+1}t^{\tilde \ell(s)}}{1 -q^{a(s)}t^{\ell(s)+1}}
\end{equation}
and for which the l.h.s. of \eqref{eqidentity} is a reproducing kernel.
Indeed, defining $\langle \cdot, \cdot \rangle'$ as
$$
\langle f(x;q,t), g(x;q,t) \rangle' =  t^{-{m \choose 2}} \left \langle f(x;q,t),   \overline{  \tau_1 \cdots \tau_m K_{\omega_m} \bar T_{\omega_m} g(x;q,t)} \right \rangle_m
$$
where $\overline{h(x;q,t)}=h(x;q^{-1},t^{-1})$ for any $h(x;q,t) \in R_m$, we have from Proposition~\ref{propoinv} that \eqref{eqscalar} holds. Note that, for $i=1,\dots,m-1$, the adjoint of $T_i$ with respect to this sesquilinear scalar product is $\bar T_i$ while, as was seen in \cite{L}, $T_i$ is self-adjoint with respect to the scalar product  $\langle \cdot, \cdot \rangle_m$. 
\end{remark}

Working in $N=m$ variables, we obtain a Cauchy-type identity for the non-symmetric Macdonald polynomials.
Note that since there is no restriction on $m$, the result holds for any number of variables.
\begin{proposition} \label{cauchy}
  For $\bar x=(x_1,\dots,x_m)$ and  $\bar y=(x_1,\dots,x_m)$, we have 
  \begin{equation} \label{cauchynew}
 K_0(\bar x,q \bar y)  \left[ \prod_{1\leq i <j \leq m}  \frac{1-tx_i y_j}{1-x_i y_j} \right]  \left[ \prod_{1\leq i \leq m}  \frac{1}{1-x_i y_i} \right] = \sum_{\eta \in \mathbb Z^m_{\geq 0}} a_\eta(q,t) E_\eta(\bar x;q,t) E_\eta(\bar y;q^{-1},t^{-1})
  \end{equation}
 where
 $$
a_\eta(q,t)^{-1} = \prod_{s \in \eta} \frac{1 -q^{\tilde a(s)+1}t^{\tilde \ell(s)}}{1 -q^{\tilde a(s)+1}t^{\ell(s)+1}}
 $$
\end{proposition}  
\begin{proof} From Proposition~\ref{propostable}, we have that $P_{(\pmb a;\lambda)}(x_{(m)};q,t)=0$ if $\lambda \neq \emptyset$. Moreover, when $\lambda=\emptyset$, we get from
  \eqref{eqPE} that $P_{(\pmb a;\emptyset)}(x_{(m)};q,t)=E_{\eta}(x_{(m)};q,t)$  for $\eta=\pmb a=(a_1,\dots,a_m)$. The result is then immediate from Corollary~\ref{coroKm} since the diagram of the composition $\eta$ is equal to that of the $m$-partition $(\pmb a;\emptyset)$, and since $a(s)=\tilde a(s)+1$ given that every row of the diagram of $\eta$ ends with a circle.
\end{proof}  
\begin{remark} \label{remarkcauchy} A different Cauchy-type identity for the non-symmetric Macdonald polynomials was provided in \cite{MN}.  For $\bar x=(x_1,\dots,x_m)$ and  $\bar y=(x_1,\dots,x_m)$, it reads in our language as
\begin{equation} \label{cauchyold}
K_0(\bar x,\bar y)  \left[ \prod_{1\leq j <i \leq m}  \frac{1-x_i y_j}{1-t x_i y_j} \right]  \left[ \prod_{1\leq i \leq m}  \frac{1}{1-t x_i y_i} \right] 
=  \sum_{\eta \in \mathbb Z^m_{\geq 0}}   a_\eta(q,t) E_\eta(\bar x;q,t) E_\eta(\bar y;q^{-1},t^{-1})
\end{equation}
 We will now see that
 \eqref{cauchynew} and \eqref{cauchyold} are essentially equivalent by recovering \eqref{cauchynew} from \eqref{cauchyold}. First observe that
 $$
 a_\eta(q^{-1},t^{-1})=  t^{-|\eta|} a_\eta(q,t) 
 $$
 and 
 $$
 E_\eta(t \bar x;q^{-1},t^{-1})=t^{|\eta|} E_\eta(\bar x ;q^{-1},t^{-1})
 $$
We also have using \eqref{eqkernel0} that
$$
K_0(\bar x,\bar y) \Big |_{(q,t) \mapsto (q^{-1},t^{-1})}= \sum_\lambda z_{\lambda}(q^{-1},t^{-1})^{-1}  p_\lambda(\bar x) p_\lambda(\bar y) =  \sum_\lambda z_{\lambda}(q,t)^{-1}  p_\lambda(q \bar x/t) p_\lambda(\bar y) = K_0(q \bar x/t,\bar y)
$$
 Letting  $(q,t) \mapsto (q^{-1},t^{-1})$ in \eqref{cauchyold} followed by $x_i \mapsto tx_i$, for $i=1,\dots,m$, 
 thus yields 
\begin{equation*} 
K_0(q \bar x,\bar y)    \left[ \prod_{1\leq j <i \leq m}  \frac{1-tx_i y_j}{1-x_i y_j} \right]  \left[ \prod_{1\leq i \leq m}  \frac{1}{1-x_i y_i} \right] 
  =  \sum_{\eta \in \mathbb Z^m_{\geq 0}}    a_\eta(q,t) E_\eta(\bar x ;q^{-1},t^{-1}) E_\eta(\bar y;q,t)
\end{equation*}
Interchanging $\bar x$ and $\bar y$ then leads to \eqref{cauchynew}.
\end{remark}

Finally, we get from Proposition~\ref{propounitriang} and Theorem~\ref{theoortho} that the $m$-symmetric Macdonald polynomials also have and orthogonality/unitriangularity  characterization akin to that of the usual Macdonald polynomials.
\begin{proposition} \label{propdefmacdo} The $m$-symmetric Macdonald polynomials form the 
  unique basis $\{ P_\Lambda(x;q,t)\}_\Lambda$ of $R_m$ such that
  \begin{enumerate}
\item $\displaystyle{\langle P_\Lambda(x;q,t)\, ,\, P_\Omega(x;q,t)  \rangle_m = 0 {\rm ~if~} \Lambda\neq\Omega}  $ \quad {\rm (orthogonality)}    
\item $\displaystyle{P_\Lambda(x;q,t)= m_\Lambda + \sum_{\Omega < \Lambda} d_{\Lambda \Omega}(q,t) \, m_\Omega}$ \quad {\rm (unitriangularity)}
  \end{enumerate}
for certain coefficients $d_{\Lambda \Omega}(q,t) \in \mathbb Q(q,t)$. We recall that the dominance order on $m$-partitions was defined in \eqref{deforder}.  
\end{proposition}

\begin{appendix} 
  \section{Proof of Theorem~\ref{theoinclusion}} \label{appA}
  \begin{proof}
  We will prove the result in $N$ variables. The case $N \to \infty$ will then be immediate.
  Recall that
  $$
P_\Lambda(x_1,\dots,x_N;q,t) = \frac{1}{u_{\Lambda,N}(t)} \mathcal S_{m+1,N}^t E_{\eta_{\Lambda,N}}(x_1,\dots,x_N;q,t)
$$

First consider the case $\Lambda=(\pmb a;b^{N-m})$ for any $b \geq 0$.
The case $b=0$ was proven in \cite{L} (the proof is essentially the same as the one we will provide in the case $b>0$) so we can assume that $b >0$.
In that case, there are two possibilities for the $(m+1)$-circle.  It can be added in the uppermost row of size $b$ to give $\Lambda^b=(\pmb a,b;b^{N-m-1})$ or in a row of size 0 to give $\Lambda^0=(\pmb a,0;b^{N-m})$. But from Proposition~\ref{propostable}, 
we have that $P_{\Lambda^0}(x;q,t)=0$  in $N$ variables given that $b>0$ by assumption. The only possibility is thus $\Lambda^b$.  Observe that in this case, 
 all the rows above that of the $(m+1)$-circle in $\Lambda^b$ end with a circle and thus do not contribute to $\psi_{\Lambda^b/\Lambda}$.  As such, we need to show that
 $i(P_\Lambda)=P_{\Lambda^b}$.

Recall from \eqref{eqLR} that $\mathcal S_{m+1,N}^t= \mathcal S_{m+2,N}^t \mathcal R_{m+1,N}$.  Using $\eta_{\Lambda,N}=(a_1,\dots,a_m,b^{N-m})$, we obtain from \eqref{property2} that
$$
\mathcal R_{m+1,N} E_{\eta_{\Lambda,N}}(x_1,\dots,x_N;q,t)= (1+t+t^2+\cdots+t^{N-m-1}) E_{\eta_{\Lambda,N}}(x_1,\dots,x_N;q,t)
$$
Hence, in order to prove that $i(P_\Lambda)=P_{\Lambda^b}$, 
we need to 
show that
$$
\frac{1}{u_{\Lambda,N}(t)} \mathcal S_{m+2,N}^t (1+t+t^2+\cdots+t^{N-m-1})E_{\eta_{\Lambda,N}}=\frac{1}{u_{\Lambda^b,N}(t)} \mathcal S_{m+2,N}^t E_{\eta_{\Lambda^b,N}}
$$
But this is easily seen to be the case given that $\eta_{\Lambda,N}=\eta_{\Lambda^b,N}$ and
$$
\frac{(1+t+\dots+t^{N-m-1})}{u_{\Lambda,N}(t)}= \frac{(1+t^{-1}+\dots+t^{-(N-m-1)})}{u_{\Lambda,N}(t)}\frac{t^{(N-m)(N-m-1)/2}}{t^{(N-m-1)(N-m-2)/2}} = \frac{1}{u_{\Lambda^b,N}(t)}
$$

We now consider the general case.  We let $\eta_{\Lambda,N}=(a_1,\dots,a_m,\lambda_{N-m},\dots,\lambda_2,\lambda_1)$, where we consider that $\lambda_{N-m}$ can be equal to 0.  Observe that we can assume that $\lambda_{N-m}<\lambda_1$ since the case $\lambda_{N-m}=\lambda_1$ corresponds to the case $\Lambda=(\pmb a;b^{N-m})$, which was already established.

We will proceed by induction on $N-m$.  The theorem  is readily seen to hold when $N=m+1$ since in this case $\Lambda$ can only be of the form $\Lambda=(\pmb a;b)$, which has been seen to hold.  In the following, we will denote by $\lambda \setminus b$ the partition obtained by removing one row of length $b$ from $\lambda$ (similarly  $\lambda \setminus \{b ,c \}$ will stand for the partition obtained by removing the entries $b$ and $c$ from $\lambda$).

We will first show that the theorem holds for  the $(m+1)$-partitions $\Omega=(\pmb a, b ; \lambda \setminus b)$ such that $b \neq \lambda_{N-m}$. For simplicity, we will let $s=\lambda_{N-m}$
and use the notation $\Lambda^+=(\pmb a,s;\lambda \setminus s)$
and $\Omega^+=(\pmb a,s, b ; \lambda \setminus \{b,s\})$.
By induction on  $N-m$, we have that
\begin{equation} \label{eqinduction}
\frac{1}{u_{\Lambda^+,N}(t)} \mathcal S_{m+2,N}^t E_{\eta_{\Lambda^+,N}}= \sum_{\Delta}
\psi_{\Delta/\Lambda^+} \frac{1}{u_{\Delta,N}(t)} \mathcal S_{m+3,N}^t E_{\eta_{\Delta,N}}
\end{equation}
Our goal is to show that
\begin{equation}
\frac{1}{u_{\Lambda,N}(t)} \mathcal S_{m+1,N}^t E_{\eta_{\Lambda,N}}\Big|_{P_{\Omega}}=  \psi_{\Omega^+/\Lambda^+}
\end{equation}
This will prove our claim  since $\psi_{\Omega^+/\Lambda^+}=\psi_{\Omega/\Lambda}$ given
that the only difference between $\Lambda^+$ and $\Lambda$ is the extra $(m+1)$-circle in row $s$ (which does no affect $\psi_{\Omega^+/\Lambda^+}$).

Using $\mathcal R_{m+1,N}=1 +T_{m+1} \mathcal R_{m+2,N}$, we get
\begin{equation} \label{eqst}
\mathcal S_{m+1,N}^t=\mathcal S_{m+2,N}^t \mathcal R_{m+1,N}= \mathcal S_{m+2,N}^t  + \mathcal S_{m+2,N}^t T_{m+1} \mathcal R_{m+2,N}
\end{equation}
Now, when acting on $E_{\eta_{\Lambda,N}}$, the operator $\mathcal S_{m+2,N}^t$
will produce a linear combination of $E_\nu$'s such that the first $m+1$ entries of $\nu$ are $a_1,\dots,a_m,s$ with $s\neq b$. This implies that  the term $\mathcal S_{m+2,N}^t$ in the r.h.s. of \eqref{eqst} will not contribute to the coefficient $\psi_{\Omega/\Lambda}$. Therefore, using
\begin{align} 
 \mathcal S_{m+2,N}^t T_{m+1} \mathcal R_{m+2,N}
 & =\mathcal L_{m+2,N} \mathcal S_{m+3,N}^t T_{m+1} \mathcal R_{m+2,N}  \nonumber \\
 & = \mathcal L_{m+2,N} T_{m+1} \mathcal S_{m+3,N}^t \mathcal R_{m+2,N}= \mathcal L_{m+2,N} T_{m+1} \mathcal S_{m+2,N}^t \nonumber
\end{align}
we obtain from \eqref{eqst} that
$$
\frac{1}{u_{\Lambda,N}(t)} \mathcal S_{m+1,N}^t E_{\eta_{\Lambda,N}}\Big|_{P_{\Omega}}=
\frac{1}{u_{\Lambda,N}(t)} \mathcal L_{m+2,N} T_{m+1} \mathcal S_{m+2,N}^t  E_{\eta_{\Lambda,N}} \Big|_{P_{\Omega}}
$$
Given that $\eta_{\Lambda,N}=\eta_{\Lambda^+,N}$, we then obtain
$$
\frac{1}{u_{\Lambda,N}(t)} \mathcal S_{m+1,N}^t E_{\eta_{\Lambda,N}}\Big|_{P_{\Omega}}=
\frac{u_{\Lambda^+,N}(t)}{u_{\Lambda,N}(t)} \mathcal L_{m+2,N} T_{m+1} \left( \frac{1}{u_{\Lambda^+,N}(t)} \mathcal S_{m+2,N}^t E_{\eta_{\Lambda^+,N}} \right)\Bigg|_{P_{\Omega}} 
$$
From \eqref{eqinduction}, we then have
$$
\frac{1}{u_{\Lambda,N}(t)} \mathcal S_{m+1,N}^t E_{\eta_{\Lambda,N}}\Big|_{P_{\Omega}}=  \sum_{\Delta}
\psi_{\Delta/\Lambda^+} \frac{1}{u_{\Delta,N}(t)} \frac{u_{\Lambda^+,N}(t)}{u_{\Lambda,N}(t)} \mathcal L_{m+2,N} T_{m+1}  \mathcal S_{m+3,N}^t E_{\eta_{\Delta,N}}\Big|_{P_{\Omega}}
$$
Now, we observe that the only way to
 generate $P_\Omega$ is to act with $T_{m+1}$ on  $\mathcal S_{m+3,N}^t E_{\eta_{\Omega^+,N}}$ (otherwise the resulting indexing composition will not have a $b$ in the $(m+1)$-th position).  From \eqref{property2}, we see that
$$
T_{m+1} \mathcal S_{m+3,N}^t E_{\eta_{\Omega^+,N}} = \mathcal S_{m+3,N}^t T_{m+1} E_{\eta_{\Omega^+,N}} = t \mathcal S_{m+3,N}^t  E_{\eta_{\Omega,N}} + A(q,t) \mathcal S_{m+3,N}^t  E_{\eta_{\Omega^+,N}}   
$$
for some coefficient $A(q,t)$. This yields
\begin{align*}
\frac{1}{u_{\Lambda,N}(t)} \mathcal S_{m+1,N}^t E_{\eta_{\Lambda,N}}\Big|_{P_{\Omega}}& = \psi_{\Omega^+/\Lambda^+}
\frac{u_{\Lambda^+,N}(t)}{u_{\Lambda,N}(t)} \frac{u_{\Omega,N}(t)}{u_{\Omega^+,N}(t)} \mathcal L_{m+2,N}  \left( \frac{t}{u_{\Omega,N}(t)} \mathcal S_{m+3,N}^t E_{\eta_{\Omega,N}} \right) \Bigg|_{P_{\Omega}}\\&= \psi_{\Omega^+/\Lambda^+}
\frac{u_{\Lambda^+,N}(t)}{u_{\Lambda,N}(t)} \frac{u_{\Omega,N}(t)}{u_{\Omega^+,N}(t)}  \left( \frac{t}{u_{\Omega,N}(t)} \mathcal S_{m+2,N}^t E_{\eta_{\Omega,N}} \right) \Bigg|_{P_{\Omega}} \\
& =  t\psi_{\Omega^+/\Lambda^+}
\frac{u_{\Lambda^+,N}(t)}{u_{\Lambda,N}(t)}  \frac{u_{\Omega,N}(t)}{u_{\Omega^+,N}(t)}
\end{align*}
The result thus holds since it is not too difficult to show that
$$
t
\frac{u_{\Lambda^+,N}(t)}{u_{\Lambda,N}(t)}  \frac{u_{\Omega,N}(t)}{u_{\Omega^+,N}(t)}=1
$$
given that
$$
\frac{u_{\Lambda^+,N}(t)}{u_{\Lambda,N}(t)}=\frac{[n-1]_{1/t}!}{t^{N-m-1} [n]_{1/t}!} \qquad {\rm and} \qquad   \frac{u_{\Omega,N}(t)}{u_{\Omega^+,N}(t)} =  \frac{t^{N-m-2} [n]_{1/t}!}{[n-1]_{1/t}!}
$$
with $n$  the number of occurrences of $s$ in $\lambda_{N-m},\dots,\lambda_1$.

Finally, we have to consider the case where 
the $(m+1)$-partition $\Omega=(\pmb a, b ; \lambda \setminus b)$ is such that $b = \lambda_{N-m}$. This case turns out to be somewhat more complicated.
This time, we use the relation
$$
\mathcal R_{m+1,N}=  \mathcal R_{m+1,N-1}(1+T_{N-1})  - \mathcal R_{m+1,N-2}T_{N-1}
$$
to get from \eqref{eqLR} that
\begin{align}
\mathcal S_{m+1,N}^t & =  \mathcal S_{m+2,N}^t \bigl(  \mathcal R_{m+1,N-1}(1+T_{N-1})  - \mathcal R_{m+1,N-2}T_{N-1}  \bigr) \nonumber \\
& = \mathcal L_{m+2,N}' \mathcal S_{m+2,N-1}^t   \mathcal R_{m+1,N-1} (1+ T_{N-1}) -  \mathcal L_{m+2,N}'  \mathcal L_{m+2,N-1}' \mathcal S_{m+2,N-2}^t  \mathcal R_{m+1,N-2} T_{N-1}  \nonumber \\
& =   \mathcal L_{m+2,N}' \mathcal S_{m+1,N-1}^t   + \mathcal L_{m+2,N}' \mathcal S_{m+1,N-1}^t T_{N-1}-  \mathcal L_{m+2,N}'  \mathcal L_{m+2,N-1}' \mathcal S_{m+1,N-2}^t   T_{N-1}
\label{eq12}
\end{align}
We first establish the result when $\lambda_1=\lambda_2$.  In this case, $T_{N-1} E_{\eta_{\Lambda,N}}= t  E_{\eta_{\Lambda,N}}$, which implies from the previous equation that
 \begin{align}
    \frac{1}{u_{\Lambda,N}(t)} \mathcal S_{m+1,N}^t E_{\eta_{\Lambda,N}}   
   =
\frac{(1+t)}{u_{\Lambda,N}(t)}  \mathcal L_{m+2,N}'  \mathcal S_{m+1,N-1}^t E_{\eta_{\Lambda,N}}  - \frac{t}{u_{\Lambda,N}(t)}  \mathcal L_{m+2,N}'  \mathcal L_{m+2,N-1}'  \mathcal S_{m+1,N-2}^t E_{\eta_{\Lambda,N}}   \label{eqbi}
 \end{align}
In order to use induction, we will need the relations
\begin{equation}  \label{eqphiq} 
   \Phi_q \mathcal S_{m+2,N}^t 
   = \mathcal S_{m+1,N-1}^t \Phi_q \qquad {\rm and} \qquad
\Phi_q^2 \mathcal S_{m+3,N}^t = \mathcal S_{m+1,N-2}^t \Phi_q^2
\end{equation}
where  $\Phi_q$, which was defined in \eqref{eqPhi}, is such that
$\Phi_q  E_{\eta_{\hat \Lambda,N}}= t^{r-N} E_{\eta_{\Lambda,N}}$ with
$\hat \Lambda=(\lambda_1-1,\pmb a; \lambda \setminus \lambda_1)$
and $r$ the row in the diagram of $\eta_{\hat \Lambda,N}$ corresponding to the entry $\lambda_1-1$ (the highest row of size $\lambda_1-1$ in the diagram of $\eta_{\hat \Lambda,N}$).  The first term in the r.h.s. of \eqref{eqbi} thus gives
$$
\frac{(1+t)}{u_{\Lambda,N}(t)}  \mathcal L_{m+2,N}'  \mathcal S_{m+1,N-1}^t E_{\eta_{\Lambda,N}}  = \frac{t^{-r+N}(1+t)}{u_{\Lambda,N}(t)}  \mathcal L_{m+2,N}' \Phi_q \mathcal S_{m+2,N}^t E_{\eta_{\hat \Lambda,N}} 
$$
Hence, by induction on $N-m$ (using a similar expression as in \eqref{eqinduction}) we have that the terms $\mathcal S_{m+3,N}^t E_{\eta_{\hat \Delta,N}}$ that can appear in  $\mathcal S_{m+2,N}^t  E_{\eta_{\hat \Lambda,N}}$ are such that $\hat \Delta = (\lambda_1-1,\pmb a,\lambda_i; \lambda \setminus \{ \lambda_1 , \lambda_i\})$.
Letting $\Delta= (\pmb a,\lambda_i; \lambda \setminus  \lambda_i) $, we observe
using $ \mathcal L_{m+2,N}'  \mathcal S_{m+2,N-1}^t=  \mathcal S_{m+2,N}^t$ 
that
$$
 \mathcal L_{m+2,N}' \Phi_q  \mathcal S_{m+3,N}^t E_{\eta_{\hat \Delta,N}} =  \mathcal L_{m+2,N}' \mathcal S_{m+2,N-1}^t \Phi_q E_{\eta_{\hat \Delta,N}} = t^{r-N}    \mathcal S_{m+2,N}^t E_{\eta_{\Delta,N}}  
$$
since the row 
in the diagram of $\eta_{\hat \Delta,N}$ corresponding to the entry $\lambda_1-1$
is the same as that of  the entry $\lambda_1-1$ in the diagram of $\eta_{\hat \Lambda,N}$. Therefore, focusing on the term $\Omega$, we have
$$
\frac{(1+t)}{u_{\Lambda,N}(t)}  \mathcal L_{m+2,N}'  \mathcal S_{m+1,N-1}^t E_{\eta_{\Lambda,N}} \Big |_{P_{\Omega}} = (1+t) \psi_{\hat \Omega/\hat \Lambda} \frac{u_{\hat \Lambda,N}(t)}{u_{\Lambda,N}(t)}  \frac{u_{\Omega,N}(t)}{u_{\hat \Omega,N}(t)} 
$$
where $\hat \Omega = (\lambda_1-1,\pmb a,s; \lambda \setminus \{ \lambda_1 , s\})$.  Doing a similar analysis for the second term in the rhs of \eqref{eqbi}, we obtain that
$$
-\frac{t}{u_{\Lambda,N}(t)}  \mathcal L_{m+2,N}'  \mathcal L_{m+2,N-1}'   \mathcal S_{m+1,N-1}^t E_{\eta_{\Lambda,N}} \Big |_{P_{\Omega}} = -t \psi_{\tilde \Omega/\tilde \Lambda}
 \frac{u_{\tilde \Lambda,N}(t)}{u_{\Lambda,N}(t)}  \frac{u_{\Omega,N}(t)}{u_{\tilde \Omega,N}(t)} 
$$
 with
  $\tilde \Lambda=(\lambda_2-1,\lambda_1-1,\pmb a; \lambda \setminus \{\lambda_1, \lambda_2\} )$ and $\tilde \Omega=(\lambda_2-1,\lambda_1-1,\pmb a,s ; \lambda \setminus \{\lambda_1, \lambda_2\})$. 
We thus have to prove that
\begin{equation} \label{eqimp1}
 (1+t) \psi_{\hat \Omega/\hat \Lambda} \frac{u_{\hat \Lambda,N}(t)}{u_{\Lambda,N}(t)}  \frac{u_{\Omega,N}(t)}{u_{\hat \Omega,N}(t)} - t \psi_{\tilde \Omega/\tilde \Lambda}
 \frac{u_{\tilde \Lambda,N}(t)}{u_{\Lambda,N}(t)}  \frac{u_{\Omega,N}(t)}{u_{\tilde \Omega,N}(t)} = \psi_{\Omega/\Lambda}
 \end{equation}
Since $s <\lambda_1=\lambda_2$ by assumption, we have that
 \begin{equation} \label{equ}
 \frac{u_{ \Omega,N}(t)}{u_{\Lambda,N}(t)} = \frac{[n-1]_{1/t}!}{t^{N-m-1}[n]_{1/t}!},\quad  \frac{u_{\hat \Lambda,N}(t)}{u_{\hat \Omega,N}(t)}= \frac{t^{N-m-2}[n]_{1/t}!}{[n-1]_{1/t}!}, \quad {\rm and } \quad  \frac{u_{\tilde \Lambda,N}(t)}{u_{\tilde \Omega,N}(t)}= \frac{t^{N-m-3}[n]_{1/t}!}{[n-1]_{1/t}!}
 \end{equation}
 with $n$ the number of occurrences of $s$ in $\lambda_{N-m},\dots,\lambda_1$.
 It thus remains  to prove that
 \begin{equation} \label{eqlast}
 \frac{(1+t)}{t} \psi_{\hat \Omega/\hat \Lambda}  - \frac{1}{t} \psi_{\tilde \Omega/\tilde \Lambda}  = \psi_{\Omega/\Lambda}
 \end{equation}
 Comparing  $\psi_{\hat \Omega/\hat \Lambda}$, $\psi_{\tilde \Omega/\tilde \Lambda}$ and  $\psi_{\Omega/\Lambda}$, we get that their factors are identical except in the
 rows corresponding to $\lambda_1$ and $\lambda_2$ in $\Lambda$ (they are consecutive rows, with that of $\lambda_2$ just above that of $\lambda_1$).
 The squares in those rows contribute
 $$
 \frac{1-q^{\lambda_1+1}t^{\ell-1}}{1-q^{\lambda_1+1}t^{\ell}}, \qquad 1 \qquad {\rm and} \qquad \frac{(1-q^{\lambda_1+1}t^{\ell-1})}{(1-q^{\lambda_1+1}t^{\ell})} \frac{(1-q^{\lambda_1+1}t^{\ell-2})}{(1-q^{\lambda_1+1}t^{\ell-1})}
 $$
respectively to $\psi_{\hat \Omega/\hat \Lambda}$, $\psi_{\tilde \Omega/\tilde \Lambda}$ and  $\psi_{\Omega/\Lambda}$,
where $\ell$ is the leg-length of the square in the row of $\lambda_2$ in the diagram of $\Omega$.  We therefore get that \eqref{eqlast} holds given that
$$
\frac{(1+t)}{t}  \frac{(1-q^{\lambda_1+1}t^{\ell-1})}{(1-q^{\lambda_1+1}t^{\ell})} - \frac{1}{t}= \frac{1-q^{\lambda_1+1}t^{\ell-2}}{1-q^{\lambda_1+1}t^{\ell}} =
\frac{(1-q^{\lambda_1+1}t^{\ell-1})}{(1-q^{\lambda_1+1}t^{\ell})} \frac{(1-q^{\lambda_1+1}t^{\ell-2})}{(1-q^{\lambda_1+1}t^{\ell-1})}
$$

Finally, we need to prove the result when $\lambda_1>\lambda_2$.  In this case,
$$
T_{N-1} E_{\eta_{\Lambda,N}} = t E_{s_{N-1}(\eta_{\Lambda,N})} + \frac{(t-1)}{1-q^{\lambda_1-\lambda_2}t^r} E_{\eta_{\Lambda,N}} 
$$
where $r$ is the difference between the row of $\lambda_2$ and that of $\lambda_1$ in the diagram associated to $\eta_{\Lambda,N}$.
Using  \eqref{eq12}, we obtain this time
 \begin{align*}
\frac{1}{u_{\Lambda,N}(t)} \mathcal S_{m+1,N}^t E_{\eta_{\Lambda,N}}   & =
\frac{(t-q^{\lambda_1-\lambda_2}t^r)}{u_{\Lambda,N}(t)(1-q^{\lambda_1-\lambda_2}t^r)}  \mathcal L_{m+2,N}'  \mathcal S_{m+1,N-1}^t E_{\eta_{\Lambda,N}} \\
& \qquad -
\frac{(t-1)}{u_{\Lambda,N}(t)(1-q^{\lambda_1-\lambda_2}t^r)}   \mathcal L_{m+2,N}'  \mathcal L_{m+2,N-1}'  \mathcal S_{m+1,N-2}^t E_{\eta_{\Lambda,N}}   \\
& \qquad + \frac{t}{u_{\Lambda,N}(t)}  \mathcal L_{m+2,N}'  \mathcal S_{m+1,N-1}^t E_{s_{N-1}(\eta_{\Lambda,N})} 
\\
& \qquad  
- \frac{t}{u_{\Lambda,N}(t)}  \mathcal L_{m+2,N}'  \mathcal L_{m+2,N-1}'  \mathcal S_{m+1,N-2}^t E_{s_{N-1}(\eta_{\Lambda,N})}  
 \end{align*}
Using \eqref{eqphiq} and doing a similar analysis as in the $\lambda_1=\lambda_2$ case, we obtain by induction that
$$
\frac{(t-q^{\lambda_1-\lambda_2}t^r)}{u_{\Lambda,N}(t)(1-q^{\lambda_1-\lambda_2}t^r)}   \mathcal L_{m+2,N}'  \mathcal S_{m+1,N-1}^t E_{\eta_{\Lambda,N}} \Big |_{P_{\Omega}} = \frac{(t-q^{\lambda_1-\lambda_2}t^r)}{(1-q^{\lambda_1-\lambda_2}t^r)}  \psi_{\hat \Omega/\hat \Lambda} \frac{u_{\hat \Lambda,N}(t)}{u_{\Lambda,N}(t)}  \frac{u_{\Omega,N}(t)}{u_{\hat \Omega,N}(t)} 
$$
and that
$$
- 
\frac{(t-1)}{u_{\Lambda,N}(t)(1-q^{\lambda_1-\lambda_2}t^r)}   \mathcal L_{m+2,N}'  \mathcal L_{m+2,N-1}'   \mathcal S_{m+1,N-1}^t E_{\eta_{\Lambda,N}} \Big |_{P_{\Omega}} =  -
\frac{(t-1)}{(1-q^{\lambda_1-\lambda_2}t^r)}  \psi_{\tilde \Omega/\tilde \Lambda}
 \frac{u_{\tilde \Lambda,N}(t)}{u_{\Lambda,N}(t)}  \frac{u_{\Omega,N}(t)}{u_{\tilde \Omega,N}(t)} 
$$
 Letting $\hat \Gamma=(\lambda_2-1,\pmb a;\lambda \setminus \lambda_2)$, $\hat \Gamma^s=(\lambda_2-1,\pmb a,s;\lambda \setminus \lambda_2)$,  
 $\tilde \Gamma=(\lambda_1-1,\lambda_2-1,\pmb a;\lambda \setminus \{ \lambda_1,\lambda_2\})$, and $\tilde \Gamma^s=(\lambda_1-1,\lambda_2-1,\pmb a,s;\lambda \setminus \{ \lambda_1,\lambda_2\})$  as well as making use of the relation (see Lemma~55 in \cite{L})
 $$
\mathcal S^t_{m+1,N} E_{s_{N-1}(\eta_{\Omega,N})} = \frac{(1-q^{\lambda_1-\lambda_2}t^{r+1})}{t(1-q^{\lambda_1-\lambda_2}t^r)} E_{\eta_{\Omega,N}}
 $$
 we get similarly by induction that
 $$
 \frac{t}{u_{\Lambda,N}(t)}  \mathcal L_{m+2,N}'  \mathcal S_{m+1,N-1}^t E_{s_{N-1}(\eta_{\Lambda,N})} \Big |_{P_{\Omega}} = \frac{(1-q^{\lambda_1-\lambda_2}t^{r+1})}{(1-q^{\lambda_1-\lambda_2}t^r)} \psi_{\hat \Gamma^s/\hat \Gamma} \frac{u_{\hat \Gamma,N}(t)}{u_{\Lambda,N}(t)}  \frac{u_{\Omega,N}(t)}{u_{\hat \Gamma^s,N}(t)} 
    $$
   and
   $$
- \frac{t}{u_{\Lambda,N}(t)}  \mathcal L_{m+2,N}'  \mathcal L_{m+2,N-1}'  \mathcal S_{m+1,N-2}^t E_{s_{N-1}(\eta_{\Lambda,N})}  \Big |_{P_{\Omega}} =-\frac{(1-q^{\lambda_1-\lambda_2}t^{r+1})}{(1-q^{\lambda_1-\lambda_2}t^r)}   \psi_{\tilde \Gamma^s/\tilde \Gamma}
\frac{u_{\tilde \Gamma,N}(t)}{u_{\Lambda,N}(t)}  \frac{u_{\Omega,N}(t)}{u_{\tilde \Gamma^s,N}(t)}
$$
We thus have to show that
\begin{align}
&  \frac{(t-q^{\lambda_1-\lambda_2}t^r)}{(1-q^{\lambda_1-\lambda_2}t^r)}  \psi_{\hat \Omega/\hat \Lambda} \frac{u_{\hat \Lambda,N}(t)}{u_{\Lambda,N}(t)}  \frac{u_{\Omega,N}(t)}{u_{\hat \Omega,N}(t)} - \frac{(t-1)}{(1-q^{\lambda_1-\lambda_2}t^r)}  \psi_{\tilde \Omega/\tilde \Lambda}
  \frac{u_{\tilde \Lambda,N}(t)}{u_{\Lambda,N}(t)}  \frac{u_{\Omega,N}(t)}{u_{\tilde \Omega,N}(t)} \nonumber \\
  & \qquad  + \frac{(1-q^{\lambda_1-\lambda_2}t^{r+1})}{(1-q^{\lambda_1-\lambda_2}t^r)} \psi_{\hat \Gamma^s/\hat \Gamma} \frac{u_{\hat \Gamma,N}(t)}{u_{\Lambda,N}(t)}  \frac{u_{\Omega,N}(t)}{u_{\hat \Gamma^s,N}(t)} -\frac{(1-q^{\lambda_1-\lambda_2}t^{r+1})}{(1-q^{\lambda_1-\lambda_2}t^r)}   \psi_{\tilde \Gamma^s/\tilde \Gamma}
\frac{u_{\tilde \Gamma,N}(t)}{u_{\Lambda,N}(t)}  \frac{u_{\Omega,N}(t)}{u_{\tilde \Gamma^s,N}(t)}=\psi_{\Omega/\Lambda} \label{eqtoprove}
\end{align}  
Since $s \leq \lambda_2$, the only case where $\tilde \Omega$, $\hat \Gamma^s$
or $\tilde \Gamma^s$ may not exist (for the lack of an extra $s$ in $\lambda$) is the case $\Lambda=(\pmb a; \lambda_1,\lambda_2)$ in  $N=m+2$ variables. Let us consider it first.    Using  \eqref{equ} with $n=1$, as well as $  \psi_{\tilde \Omega/\tilde \Lambda}=\psi_{\tilde \Gamma^s/\tilde \Gamma}= \psi_{\hat \Gamma^s/\hat \Gamma}=0$ in \eqref{eqtoprove}, we have to prove in this case that
 $$
\frac{(t-q^{\lambda_1-\lambda_2}t^r)}{t(1-q^{\lambda_1-\lambda_2}t^r)}  \psi_{\hat \Omega/\hat \Lambda} = \psi_{\Omega / \Lambda}
  $$
But this is immediate given that $ \psi_{\hat \Omega/\hat \Lambda}=1$
and
$$
 \psi_{\Omega / \Lambda}= \frac{1-q^{\lambda_1-\lambda_2}t^{r-1}}{1-q^{\lambda_1-\lambda_2}t^r}
$$

 As previously mentioned, the remaining cases will involve all terms: $\hat \Omega$, $\tilde \Omega$, $\hat \Gamma^s$
and $\tilde \Gamma^s$.   Using  \eqref{equ} together with
 $$
  \frac{u_{\hat \Gamma,N}(t)}{u_{\hat \Gamma^s,N}(t)}= \frac{t^{N-m-2}[n]_{t^{-1}}!}{[n-1]_{t^{-1}}!} \quad {\rm and } \quad  \frac{u_{\tilde \Gamma,N}(t)}{u_{\tilde \Gamma^s,N}(t)}= \frac{t^{N-m-3}[n]_{1/t}!}{[n-1]_{\frac{1}{t}}!}
  $$
in \eqref{eqtoprove},  
we have to prove this time that
  $$
\frac{(t-q^{\lambda_1-\lambda_2}t^r)}{t(1-q^{\lambda_1-\lambda_2}t^r)}  \psi_{\hat \Omega/\hat \Lambda} -\frac{(t-1)}{t^2(1-q^{\lambda_1-\lambda_2}t^r)}  \psi_{\tilde \Omega/\tilde \Lambda} + \frac{(1-q^{\lambda_1-\lambda_2}t^{r+1})}{t(1-q^{\lambda_1-\lambda_2}t^r)} \psi_{\hat \Gamma^s/\hat \Gamma} -\frac{(1-q^{\lambda_1-\lambda_2}t^{r+1})}{t^2(1-q^{\lambda_1-\lambda_2}t^r)}   \psi_{\tilde \Gamma^s/\tilde \Gamma} = \psi_{\Omega / \Lambda}
  $$
Comparing  $\psi_{\hat \Omega/\hat \Lambda}$, $\psi_{\tilde \Omega/\tilde \Lambda}$,
$\psi_{\hat \Gamma^s/\hat \Gamma}$, $\psi_{\tilde \Gamma^s/\tilde \Gamma}$,
and  $\psi_{\Omega/\Lambda}$, we get that their factors are identical except in the
 rows corresponding to $\lambda_1$ and $\lambda_2$ in $\Omega/\Lambda$.
 The squares in those rows contribute
 $$
 \frac{1-q^{\lambda_2+1}t^{l-1}}{1-q^{\lambda_2+1}t^{l}}, \qquad 1, \qquad  \frac{1-q^{\lambda_1+1}t^{r+l-1}}{1-q^{\lambda_1+1}t^{r+l}}, \qquad 1 \qquad {\rm and} \qquad \frac{(1-q^{\lambda_1+1}t^{r+l-1})}{(1-q^{\lambda_1+1}t^{r+l})} \frac{(1-q^{\lambda_2+1}t^{l-1})}{(1-q^{\lambda_2+1}t^{l})}
 $$
 respectively to $\psi_{\hat \Omega/\hat \Lambda}$, $\psi_{\tilde \Omega/\tilde \Lambda}$,
$\psi_{\hat \Gamma^s/\hat \Gamma}$, $\psi_{\tilde \Gamma^s/\tilde \Gamma}$,
and  $\psi_{\Omega/\Lambda}$,
where $l$ is the leg-length of the square in the row of $\lambda_2$ in the diagram of $\Omega$, and where we recall that $r$ is the difference between the row of $\lambda_2$ and that of $\lambda_1$
in the diagram associated to $\Lambda$ (or, equivalently, in the diagram associated to $\eta_{\Lambda,N}$).  The result then follows from the relation
\begin{align*}
&  \frac{(t-q^{\lambda_1-\lambda_2}t^r)}{t(1-q^{\lambda_1-\lambda_2}t^r)} \frac{(1-q^{\lambda_2+1}t^{l-1})}{(1-q^{\lambda_2+1}t^{l})}   -\frac{(t-1)}{t^2(1-q^{\lambda_1-\lambda_2}t^r)} \\
  & \qquad +\frac{(1-q^{\lambda_1-\lambda_2}t^{r+1})}{t(1-q^{\lambda_1-\lambda_2}t^r)}  \frac{(1-q^{\lambda_1+1}t^{r+l-1})}{(1-q^{\lambda_1+1}t^{r+l})}-\frac{(1-q^{\lambda_1-\lambda_2}t^{r+1})}{t^2(1-q^{\lambda_1-\lambda_2}t^r)}  = \frac{(1-q^{\lambda_1+1}t^{r+l-1})}{(1-q^{\lambda_1+1}t^{r+l})} \frac{(1-q^{\lambda_2+1}t^{l-1})}{(1-q^{\lambda_2+1}t^{l})}
\end{align*}
which can straightforwardly be checked using
$$
-\frac{(t-1)}{t^2(1-q^{\lambda_1-\lambda_2}t^r)} -\frac{(1-q^{\lambda_1-\lambda_2}t^{r+1})}{t^2(1-q^{\lambda_1-\lambda_2}t^r)}=-\frac{1}{t}
$$
and then
$$
\frac{(1-q^{\lambda_1-\lambda_2}t^{r+1})}{t(1-q^{\lambda_1-\lambda_2}t^r)}  \frac{(1-q^{\lambda_1+1}t^{r+l-1})}{(1-q^{\lambda_1+1}t^{r+l})}  -\frac{1}{t}= q^{\lambda_1-\lambda_2}t^{r-1}\frac{(1-t)(1-q^{\lambda_2+1}t^{l-1})}{(1-q^{\lambda_1-\lambda_2}t^r)(1-q^{\lambda_1+1}t^{r+l})}
$$
followed by
\begin{align*}
 \frac{(t-q^{\lambda_1-\lambda_2}t^r)}{t(1-q^{\lambda_1-\lambda_2}t^r)} \frac{(1-q^{\lambda_2+1}t^{l-1})}{(1-q^{\lambda_2+1}t^{l})} 
  +q^{\lambda_1-\lambda_2}t^{r-1}& \frac{(1-t)(1-q^{\lambda_2+1}t^{l-1})}{(1-q^{\lambda_1-\lambda_2}t^r)(1-q^{\lambda_1+1}t^{r+l})} \\
&  \qquad \qquad =\frac{(1-q^{\lambda_1+1}t^{r+l-1})}{(1-q^{\lambda_1+1}t^{r+l})} \frac{(1-q^{\lambda_2+1}t^{l-1})}{(1-q^{\lambda_2+1}t^{l})}
\end{align*}

\end{proof}

  \section{Proof of Proposition~\ref{propsymfinite}.} \label{appB}

      Let $H_m(t)$ be the Hecke algebra generated by $T_1,\dots,T_{m-1}$.  We define the linear antihomomorphism $\varphi_m: H_m(t) \to H_m(t)$ to be such that
\begin{equation*}
    \varphi_m(T_{i}) = T_{m-i} \qquad  {\rm for} \quad i = 1,\ldots, m-1
\end{equation*}
For any permutation $\sigma \in S_m$, we thus
have that
$\varphi_m (T_\sigma)=T_{\tilde \sigma}$ for a certain permutation $\tilde \sigma  \in S_m$ of the same length as $\sigma$. Hence
\begin{equation} \label{eqTm}
\varphi_m(T_{\omega_m}) = T_{\omega_m}
\end{equation}
given that $\omega_m$ is the unique longest permutation in $S_m$.
Moreover, it is easy to see that  $\varphi_m \circ \varphi_m$ is the identity and that
\begin{equation}
\varphi_m(\bar T_i) = \bar T_{m-i}
\end{equation}

\begin{lemma} \label{rel phi} Let  $\omega_{m-1}'=[1,m,m-1,\dots,2]$ be the longest permutation in the symmetric group on $\{2,\dots,m\}$.  Then,
 for all $i \in \{1,\ldots,m\}$, we have
 \begin{equation} \label{eqT}
         T_{\omega_m} = T_{i-1} \cdots T_1 \cdot 
         T_{\omega_{m-1}'} \cdot \varphi_m(T_{i} \cdots T_{m-1})
 \end{equation}
 and
\begin{equation} \label{eqTm2}
         \varphi_m(\bar{T}_{\omega_m}) = \varphi_m(\bar{T}_{1} \cdots \bar{T}_{i-1}) \cdot 
         \varphi_m(\bar{T}_{\omega_{m-1}'}) \cdot \bar{T}_{m-1} \cdots \bar{T}_{i}
    \end{equation}
\end{lemma}

\begin{proof} It suffices to prove \eqref{eqT} since
 \eqref{eqTm} can be obtained from \eqref{eqT} by 
 taking the inverse and then applying $\varphi_m$.  Since $\omega_m$ is of length $m(m-1)/2$,  \eqref{eqT} will hold if we can prove that
  $$
\omega_m=s_{i-1} \cdots s_1 \omega_{m-1}' s_1 \cdots s_{m-i}
$$
for $i=1,\dots,m$. The result is well known when $i=1$.
For an arbitrary $i$, we use the simple relation
$$
s_\ell \, \omega_m \, s_{m-\ell} = \omega_m
$$
successively (for $\ell=1,\dots,i-1$) on the $i=1$ case
$$
\omega_m=  \omega_{m-1}' s_1 \cdots s_{m-1}
$$
\end{proof}

For simplicity, we now define
\begin{equation} 
  \bar K_m(x,y)=  {{t^{{m\choose 2} }}} \bar T^{(x)}_{{\mathbf \omega}_m}  K_m(x,y)= K_0(x,y)  \left[ \frac{\prod_{i+j \leq m} (1-tq^{-1} x_i y_j)}{\prod_{i+j \leq m+1}(1-q^{-1} x_i y_j)} \right]
  \end{equation}

\begin{lemma} \label{lema Ti}
For $i=1,\dots,m-1$, we have
\begin{equation}
  T^{(x)}_i  \bar K_m(x_{(N)},y_{(N)}) = T^{(y)}_{m-i}  \bar K_m(x_{(N)},y_{(N)})
\end{equation}
Hence
\begin{equation} \label{eqvarphi}
  T^{(x)}_\sigma  \bar K_m(x_{(N)},y_{(N)}) = \varphi^{(y)}_m \left( T^{(y)}_{\sigma} \right)  \bar K_m(x_{(N)},y_{(N)})
\end{equation}
for any $\sigma \in S_m$.
\end{lemma}

\begin{proof}  We have that $ \bar K_m(x_{(N)},y_{(N)})$ is symmetric in $x_i$ and $x_{i+1}$ and in $y_{m-i}$ and $y_{m-i+1}$ except for the term
\begin{equation*}
    B := \dfrac{(1-t q^{-1}x_iy_{m-i})}{(1- q^{-1} x_iy_{m-i+1})(1- q^{-1} x_{i+1}y_{m-i})(1- q^{-1} x_{i}y_{m-i})}
\end{equation*}
We thus only have to prove that $T^{(x)}_i B  = T^{(y)}_{m-i} B $, or equivalently, that  $(T^{(x)}_i-t) B  = (T^{(y)}_{m-i} -t)B $.  The lemma thus holds after
checking that
$$
(T^{(x)}_i-t) B  = (T^{(y)}_{m-i} -t)B= \dfrac{q^{-1}(tx_i-x_{i+1})(ty_{m-i}-y_{m-i+1})}{(1-q^{-1}x_{i+1}y_{m-i})(1-q^{-1}x_{i}y_{m-i})(1-q^{-1}x_iy_{m-i+1})(1-q^{-1}x_{i+1}y_{m-i+1})} 
$$
\end{proof}

\begin{proof}[Proof of Proposition~\ref{propsymfinite}] Recall that our claim  is that
\begin{equation} \label{eqY}
  Y_i^{(x)} K_m(x_{(N)},y_{(N)}) =  Y_i^{(y)} K_m(x_{(N)},y_{(N)}) 
\end{equation}
  for $i=1,\dots,m$,
  and that
\begin{equation} \label{eqYY}
D^{(x)} K_m(x_{(N)},y_{(N)}) =  D^{(y)} K_m(x_{(N)},y_{(N)})
\end{equation}
For the remainder of the section, we will consider that $x=(x_1,\dots,x_N)$ and $y=(y_1,\dots,y_N)$.
We first prove \eqref{eqY}. For simplicity, we write $Y_i = a \cdot b \cdot \omega \cdot c$, where $a = T_{i} \ldots T_{m-1}$, $b = T_{m}\ldots T_{N-1}$ and $c = \bar{T}_{1} \ldots \bar{T}_{i-1}$ (note that if $i =1$ then $c = 1$). We will also let $d = T_{\omega_m}$ and use $a$ for $a^{(x)}$ and $\pmb a$ for $a^{(y)}$ (and similarly for $b,c,d$ and $\omega$).
Using this notation, \eqref{eqY} translates into
\begin{equation*}
  a \, b \, \omega \, c \, d \, \bar K_m(x,y) = \pmb a \, \pmb b\,  \pmb \omega \,  \pmb c \,  d \, \bar K_m(x,y)
\end{equation*}
But, since $d \,  \bar K_m(x,y)= \varphi^{(y)} ( {\pmb d}) \bar K_m(x,y)=  {\pmb d} \,   \bar K_m(x,y)$ from \eqref{eqTm} and  \eqref{eqvarphi}, this amounts to
\begin{equation*}
  a \, b \, \omega \, c \, d \, \bar K_m(x,y) = \pmb a \, \pmb b\,  \pmb \omega \,  \pmb c \, \pmb d \, \bar K_m(x,y)
\end{equation*}
which we can rewrite as
\begin{equation*}
   \bar  K_m(x,y) = \pmb a \,  \pmb b \,  \pmb \omega \, \pmb c \, \pmb  d \, \bar{d} \,  \bar{c} \,  \bar{\omega} \,  \bar{b} \,  \bar{a} \,  \bar  K_m(x,y)
\end{equation*}
where $\bar a$ stands for the inverse of $a$ (and similarly for the other terms).  From Lemma~\ref{lema Ti}, this becomes
\begin{equation*}
   \bar  K_m(x,y) = \pmb a \,  \pmb b \, \pmb  \omega \, \pmb  c \, \pmb  d   \varphi_m^{(y)}(\bar {\pmb {a}}) \, \bar{d} \,  \bar{c} \,  \bar{\omega} \,  \bar{b} \,  \bar  K_m(x,y)
\end{equation*}
It was proven in \cite{BDLM2} that  $\bar{\omega} \, 
\bar{b} \,  \bar  K_m(x,y)
= \bar{\pmb \omega} \, \bar{\pmb b}\,   \bar  K_m(x,y)$ (see the symmetry of $G(x,y)$ in (140) therein).  Therefore, we have to prove, using Lemma~\ref{lema Ti}, that
\begin{align*}
  \bar  K_m(x,y) &
  = \pmb a \, \pmb  b \, \pmb \omega \,  \pmb c \, \pmb  d \,    \varphi_m^{(y)}(\bar{\pmb a})  \, \bar{\pmb \omega} \, \bar{\pmb b} \, \bar{d} \,  \bar{c} \,  \bar  K_m(x,y)    \\
  &
  = \pmb a \, \pmb b \, \pmb  \omega \, \pmb c \, \pmb  d \,    \varphi_m^{(y)}(\bar{\pmb a}) \,  \bar{\pmb \omega} \, \bar{\pmb b} \,   \varphi_m^{(y)}(\bar{\pmb c}) \, \bar{d} \,   \bar  K_m(x,y)    \\ 
    &
  = \pmb a \, \pmb b \,  \pmb \omega \, \pmb c \, \pmb  d \,    \varphi_m^{(y)}(\bar{\pmb a}) \,  \bar{\pmb \omega} \, \bar{\pmb b} \,   \varphi_m^{(y)}(\bar{\pmb c}) \, \bar{\pmb d} \,    \bar  K_m(x,y)
\end{align*}
since, as we have seen before,  $\bar d \,  \bar K_m(x,y)= \varphi^{(y)} (\bar {\pmb d}) \bar K_m(x,y)= \bar {\pmb d} \,   \bar K_m(x,y)$ from \eqref{eqTm} and  \eqref{eqvarphi}.  The equality will thus follow if we can prove that
\begin{equation*}
    a \, b \, \omega \, c \, d \, \varphi_m(\bar a) \,  \bar{\omega}\,  \bar{b} \, \varphi_m(\bar {c}) \, \bar {d} = 1
\end{equation*}
where 1 stands for the identity operator.  Now, we use Lemma~\ref{rel phi} to get $d=\bar c \, e \,  \varphi_m(a)$ and $\bar d=\varphi_m(\bar d)= \varphi_m(c) \varphi_m(\bar e) \bar a$, where $e = T_{\omega_{m-1}'}$. This yields
\begin{equation} \label{eqq}
\begin{array}{ll}
 a \, b \, \omega \, c \, d \, \varphi_m(\bar a) \,  \bar{\omega} \, \bar{b} \, \varphi_m(\bar {c})\,  \bar {d}
  &= a \,  b \, \omega \, c  \, \bigl(\bar {c} \, e \, \varphi_m(a)\bigr) \,  \varphi_m(\bar{a}) \, \bar{\omega} \, \bar{b} \, \varphi_m(\bar {c}) \, \bigl(\varphi_m(c) \,  \varphi_m(\bar {e}) \, \bar {a}\bigr) \\
   &= a \, b \, \omega \,    e \,  \bar{\omega} \, \bar{b} \, \varphi_m(\bar {e}) \, \bar {a}
\end{array}
\end{equation}
Finally, we have that $\varphi_m(\bar {e}) = \varphi_m(\bar {T}_{\omega_{m-1}'}) = \bar {T}_{\omega_{m-1}} = \bar {e}'$ since $\varphi_m$ changes the set $\{T_2,\dots T_{m-1}\}$ to $\{T_1,\ldots,T_{m-2}\}$. Using $\bar {b} \,  \bar{e}' = \bar{e}'\, \bar{b}$
and $\bar{\omega} \, \bar{e}' = \bar{e} \, \bar{\omega}$, we obtain that
$$
 a \, b \, \omega   \, e \,  \bar{\omega}\,  \bar{b}\, \varphi_m(\bar {e}) \, \bar {a}=  a \, b \, \omega  \,  e\,  \bar{\omega} \,  \bar e' \, \bar{b}   \, \bar {a} =  a \, b \, \omega  \,  e \, \bar e \, \bar{\omega}  \,  \bar{b}   \, \bar {a} =1
$$
as wanted.

For \eqref{eqYY}, we have to prove that
\begin{equation*} 
    D^{(x)} T^{(x)}_{\omega_m}  \bar  K_m(x,y) = D^{(y)} T^{(y)}_{\omega_m}  \bar  K_m(x,y)   
    \end{equation*}
Since $ T_{\omega_m}^{(x)} $ commutes with $D^{(x)}$ and  $T^{(y)}_{\omega_m}  \bar  K_m(x,y)= T^{(x)}_{\omega_m}  \bar  K_m(x,y)$     from \eqref{eqTm} and  \eqref{eqvarphi}, we only need to show that
\begin{equation*} \label{simetria m}
    D^{(x)}  \bar  K_m(x,y) = D^{(y)}   \bar  K_m(x,y)   
    \end{equation*}
Now, $D^{(x)}$ commutes with $\mathcal S^{t (x)}_{m+1,N}$ and
$$
\mathcal S^{t (x)}_{m+1,N}  \bar  K_m(x,y)  \propto \bar  K_m(x,y)  
$$
since $\bar K_m(x_{(N)},y_{(N)})$ is symmetric in $x_{m+1},\dots,x_N$ (and similarly when $x$ is replaced by $y$).  It is thus equivalent to show that
$$
\mathcal S^{t (x)}_{m+1,N}  D^{(x)} \bar  K_m(x,y) = \mathcal S^{t (y)}_{m+1,N}  D^{(y)} \bar  K_m(x,y) 
$$
From the definition of $D$, the result will thus follow if we can show that
$$
\mathcal S^{t (x)}_{m+1,N}  Y_i^{(x)} \bar  K_m(x,y) = \mathcal S^{t (y)}_{m+1,N}  Y_i^{(y)} \bar  K_m(x,y) 
$$
for $i=m+1,\dots,N$.  Using $\mathcal S^{t (x)}_{m+1,N} T_j^{(x)}=t \mathcal S^{t (x)}_{m+1,N}  $ and $\bar T_j^{(x)} \bar  K_m(x,y) =t^{-1} \bar  K_m(x,y) $  for $j=m+1,\dots,N$, we obtain in those cases that  (up to a power of $t$)
$$
\mathcal S^{t (x)}_{m+1,N}  Y_i^{(x)} \bar  K_m(x,y) \propto \mathcal S^{t (x)}_{m+1,N}  \omega^{(x)} \bar T_1^{(x)} \cdots \bar T_m^{(x)} \bar  K_m(x,y)
$$
Using the same relation with $x$  replaced by $y$, we
 thus have left to prove that
$$
\mathcal S^{t (x)}_{m+1,N}  \omega^{(x)} \bar T_1^{(x)} \cdots \bar T_m^{(x)} \bar  K_m(x,y)=\mathcal S^{t (y)}_{m+1,N}  \omega^{(y)} \bar T_1^{(y)} \cdots \bar T_m^{(y)} \bar  K_m(x,y)
$$
But this was proven in \cite{BDLM2} (see the symmetry of $L(x,y) $ in (154) therein).
\end{proof}

\section{Missing piece in the proof of Proposition~\ref{propoinv}} \label{appC}
We will prove the following lemma which was needed in the proof of Proposition~\ref{propoinv}.
\begin{lemma} For any $f\in R_m$ in $N$ variables, we have
   \begin{equation} \label{eqinv11}
Y_i^\star \left( \tau_1 \cdots\tau_m K_{\omega_m} \bar T_{\omega_m} \right) f= \left(\tau_1 \cdots \tau_m  K_{\omega_m} \bar T_{\omega_m} \right) \bar Y_i f \qquad i=1,\dots,m 
\end{equation}
and
\begin{equation} \label{eqinv22}
D^\star  \left( \tau_1 \cdots \tau_m K_{\omega_m} \bar T_{\omega_m} \right) f=  \left( \tau_1 \cdots \tau_m K_{\omega_m} \bar T_{\omega_m} \right) \bar D f 
\end{equation}
where $\bar D=\bar Y_{m+1}+\cdots +\bar Y_N + \sum_{i=m+1}^N t^{i-1}$.  
\end{lemma} 
\begin{proof}
We first prove \eqref{eqinv11} in the case $i=1$.  Observe that $\bar Y_1$ and $\tau_1 \cdots \tau_m K_{\omega_m} \bar T_{\omega_m}$ both preserve $R_m$. We can thus prove instead that
$$
Y_1^\star \left( \tau_1 \cdots\tau_m K_{\omega_m} \bar T_{\omega_m} \right) f= K_{\sigma_{m+1}}\left(\tau_1 \cdots \tau_m  K_{\omega_m} \bar T_{\omega_m} \right) \bar Y_1 f
$$
where $\sigma_{m+1}=[1,\dots,m,N,N-1,\dots,m+1]$. Hence, using the expression for $Y_1$, we need to prove that
\begin{equation*}
 T_1^\star \cdots T_{N-1}^\star s_{N-1} \cdots s_1  \tau_2 \cdots \tau_m K_{\omega_m} \bar T_{\omega_m} f=   K_{\sigma_{m+1}} \tau_1 \cdots \tau_m  K_{\omega_m} \bar T_{\omega_m} \tau_1^{-1} s_1 \dots s_{N-1} \bar T_{N-1} \cdots \bar T_1 f
\end{equation*}  
 where we use $s_i=K_{i,i+1}$ for simplicity. Using $ K_{\omega_m} \bar T_{\omega_m}=   T_{\omega_m}^\star K_{\omega_m}$ and multiplying both sides by
 $\bar T_{m-1}^\star  \cdots\bar T_1^\star$, this is equivalent to proving that
\begin{equation*}
 T_m^\star \cdots T_{N-1}^\star s_{N-1} \cdots s_1  \tau_2 \cdots \tau_m K_{\omega_m} \bar T_{\omega_{m}} f=  K_{\sigma_{m+1}}  \tau_1 \cdots \tau_m  T_{\omega_{m-1}}^\star
  K_{\omega_m}   \tau_1^{-1} s_1 \dots s_{N-1} \bar T_{N-1} \cdots \bar T_1 f
\end{equation*}  
where we used $T_{\omega_{m-1}}=\bar T_{m-1} \cdots \bar T_1 T_{\omega_{m}}$. Letting $f'=\bar T_{m-1} \cdots \bar T_1 f$,
we thus have to prove that
\begin{equation*}
 T_m^\star \cdots T_{N-1}^\star s_{N-1} \cdots s_1  \tau_2 \cdots \tau_m K_{\omega_m} \bar T_{\omega_{m-1}} f'=   K_{\sigma_{m+1}} \tau_1 \cdots \tau_m  T_{\omega_{m-1}}^\star
  K_{\omega_m}   \tau_1^{-1} s_1 \dots s_{N-1} \bar T_{N-1} \cdots \bar T_m f'
\end{equation*}  
for $f' \in R_m$. Owing to the relations $ K_{\omega_m}   \tau_1^{-1}= \tau_m^{-1} K_{\omega_m}$ and $s_{N-1} \cdots s_1  \tau_2 \cdots \tau_m=  \tau_1 \cdots \tau_{m-1} s_{N-1} \cdots s_1$, we can extract $\tau_1 \cdots \tau_{m-1}$ to the left on both sides of the equation. We thus have left to prove that 
\begin{equation*}
 T_m^\star \cdots T_{N-1}^\star s_{N-1} \cdots s_1  K_{\omega_m} \bar T_{\omega_{m-1}} f'=   K_{\sigma_{m+1}} T_{\omega_{m-1}}^\star
  K_{\omega_m}   s_1 \dots s_{N-1} \bar T_{N-1} \cdots \bar T_m f'
\end{equation*}  
With the help of $s_{m-1} \cdots s_1 K_{\omega_{m}}= K_{\omega_{m}} s_{1} \cdots s_{m-1}=K_{\omega_{m-1}}$ and $ K_{\omega_{m-1}} \bar T_{\omega_{m-1}}=   T_{\omega_{m-1}}^\star K_{\omega_{m-1}}$, this now amounts to showing that
\begin{equation*}
 T_m^\star \cdots T_{N-1}^\star s_{N-1} \cdots s_{m}  f'=  K_{\sigma_{m+1}} s_{m} \dots s_{N-1} \bar T_{N-1} \cdots \bar T_m f'
\end{equation*}  
Letting $\sigma_m=[1,\dots,m-1,N,N-1,\dots,m]$ and using the fact that $f' \in R_m$, this is seen to hold since
\begin{equation*}
\begin{split}
    T_m^\star \cdots T_{N-1}^\star s_{N-1} \cdots s_{m}  f'=
 T_m^\star \cdots T_{N-1}^\star K_{\sigma_m}  f'& =  K_{\sigma_m}  \bar T_{N-1} \cdots \bar T_{m}   f' \\ 
 &=
  K_{\sigma_{m+1}} s_{m} \dots s_{N-1} \bar T_{N-1} \cdots \bar T_m f'
  \end{split}
\end{equation*}  

We now consider the general case in \eqref{eqinv11}.  From the relation
$$
Y_i= \bar T_{i-1} \cdots \bar T_{1} Y_1 \bar T_{1} \cdots \bar T_{i-1} 
$$
we have to show that, for $1 \leq i \leq m$, we have
$$
\bigl( \bar T_{i-1}^\star \cdots \bar T_{1}^\star  Y_1^\star \bar T_{1}^\star \cdots \bar T_{i-1}^\star \bigr) \tau_1 \dots \tau_m K_{\omega_m} \bar T_{\omega_m} =  \tau_1 \dots \tau_m K_{\omega_m} \bar T_{\omega_m} \bigl( T_{i-1} \cdots  T_{1} Y_1  T_{1} \cdots  T_{i-1}  \bigr)
$$
But this is indeed the case since using the $Y_1$ case  and $T_i^\star K_{\omega_m} \bar T_{\omega_m}= K_{\omega_m} \bar T_{\omega_m} T_i$
for all $i=1,\dots,m$, we get that 
\begin{equation*}
  \begin{split}
    \bigl( \bar T_{i-1}^\star \cdots \bar T_{1}^\star  Y_1^\star \bar T_{1}^\star \cdots \bar T_{i-1}^\star \bigr) \tau_1 \dots \tau_m K_{\omega_m} \bar T_{\omega_m} 
    & =  \bar T_{i-1}^\star \cdots \bar T_{1}^\star  Y_1^\star \tau_1 \dots \tau_m   K_{\omega_m}   \bar   T_{\omega_m}  T_{1} \cdots T_{i-1} \\
    & =  \bar T_{i-1}^\star \cdots \bar T_{1}^\star  \tau_1 \dots \tau_m   K_{\omega_m}   \bar   T_{\omega_m}  \bar Y_1  T_{1} \cdots T_{i-1} \\
     & =   \tau_1 \dots \tau_m   K_{\omega_m}   \bar   T_{\omega_m}  \bigl( T_{i-1} \cdots T_{1} \bar Y_1  T_{1} \cdots T_{i-1} \bigr)
\end{split}    
\end{equation*}
Finally, proceeding as in the proof of \eqref{eqYY}, in order to prove \eqref{eqinv22} we only have to show that
\begin{equation*}
( \mathcal S^{t}_{m+1,N})^\star Y^\star_{m+1}  \tau_1 \cdots \tau_m K_{\omega_m} \bar T_{\omega_m} f=  \tau_1 \cdots \tau_m K_{\omega_m} \bar T_{\omega_m}  \mathcal S^{t}_{m+1,N} \bar Y_{m+1} f 
\end{equation*}
on any $f\in R_m$. This is equivalent to proving that
\begin{equation*}
  K_{\sigma_{m+1}} ( \mathcal S^{t}_{m+1,N})^\star Y^\star_{m+1}  \tau_1 \cdots \tau_m K_{\omega_m} \bar T_{\omega_m}f=  \tau_1 \cdots \tau_m K_{\omega_m} \bar T_{\omega_m}  \mathcal S^{t}_{m+1,N} \bar Y_{m+1} f 
\end{equation*}
since  $\tau_1 \cdots \tau_m K_{\omega_m} \bar T_{\omega_m}  \mathcal S^{t}_{m+1,N} \bar Y_{m+1}$ preserves $R_m$.
 Using $(\mathcal S^{t}_{m+1,N})^\star   T_i^\star=t^{-1} (\mathcal S^{t}_{m+1,N})^\star$ and $\bar T_i f = t^{-1} f$  for $i=m+1,\dots,N$, 
this thus  amounts to showing that
\begin{equation} \label{eqappfinal}
  \begin{split}
    \mathcal S^{t}_{m+1,N}  K_{\sigma_{m+1}}  s_{N-1} \cdots s_1 \tau_1^{-1} \bar T_1^\star \cdots  \bar T_m^\star & \tau_1 \cdots \tau_m K_{\omega_m} \bar T_{\omega_m} f\\
    & =  \tau_1 \cdots \tau_m K_{\omega_m} \bar T_{\omega_m}  \mathcal S^{t}_{m+1,N} T_m \cdots T_1 \tau_1^{-1} s_1 \cdots s_{m}  f
\end{split}    
\end{equation}
where we have used the relation $ K_{\sigma_{m+1}}  (\mathcal S^{t}_{m+1,N})^\star  =  \mathcal S^{t}_{m+1,N}   K_{\sigma_{m+1}}$.  We will now see that \eqref{eqappfinal} holds. We first use $K_{\sigma_{m+1}}  s_{N-1} \cdots s_{m+1}=K_{\sigma_{m+2}}$ to obtain
\begin{equation*}
\begin{split}
 \mathcal S^{t}_{m+1,N} K_{\sigma_{m+1}}  s_{N-1} \cdots s_1 \tau_1^{-1} \bar T_1^\star \cdots  \bar T_m^\star  &\tau_1 \cdots \tau_m   K_{\omega_m} \bar T_{\omega_m} f \\
   &  =   \mathcal S^{t}_{m+1,N} K_{\sigma_{m+2}}  s_{m} \cdots s_1 \tau_1^{-1} \bar T_1^\star \cdots  \bar T_m^\star  \tau_1 \cdots \tau_m K_{\omega_m} \bar T_{\omega_m} f \\
    & =   \mathcal S^{t}_{m+1,N}  s_{m} \cdots s_1 \tau_1^{-1} \bar T_1^\star \cdots  \bar T_m^\star  \tau_1 \cdots \tau_m K_{\omega_m} \bar T_{\omega_m} f \
\end{split}
\end{equation*}
since $K_{\sigma_{m+2}} f = f$. We then observe that 
$$ \bar T_1^\star \cdots  \bar T_m^\star  T_{\omega_m}^\star K_{\omega_{m+1}}= K_{\omega_{m+1}}  T_m \cdots  T_1  \bar T_{\omega_m'}= K_{\omega_{m+1}}  \bar T_{\omega_m}   T_m \cdots   T_1
$$ 
since $s_m \cdots s_1 \omega_m' =\omega_m s_m \cdots s_1$, where $\omega_m'=[1,m,m-1,\dots,2,m+1,\dots,N]$.
But then \eqref{eqappfinal} holds given that
\begin{equation*}
  \begin{split}
    \mathcal S^{t}_{m+1,N}  s_{m} \cdots s_1 \tau_1^{-1} \bar T_1^\star \cdots &  \bar T_m^\star  \tau_1 \cdots \tau_m K_{\omega_m} \bar T_{\omega_m} f \\
    & =   \mathcal S^{t}_{m+1,N}  s_{m} \cdots s_1 \tau_1^{-1} \bar T_1^\star \cdots  \bar T_m^\star  T_{\omega_m}^\star \tau_1 \cdots \tau_m K_{\omega_m}  f \\
    & =   \mathcal S^{t}_{m+1,N}  s_{m} \cdots s_1 \tau_1^{-1}   \bar T_1^\star \cdots  \bar T_m^\star  T_{\omega_m}^\star   \tau_1 \cdots \tau_m K_{\omega_{m+1}} s_1 \cdots s_m  f \\
    & =   \mathcal S^{t}_{m+1,N}  K_{\omega_m} \tau_{m+1}^{-1}   \bar T_{\omega_m}  T_m \cdots   T_1  \tau_2 \cdots \tau_{m+1} s_1 \cdots s_m  f \\
    & =   \mathcal S^{t}_{m+1,N}  K_{\omega_m} \tau_{1} \cdots \tau_m  \bar T_{\omega_m}  T_m \cdots   T_1   \tau_1^{-1} s_1 \cdots s_m  f \\
     & =  \tau_{1} \cdots \tau_m  K_{\omega_m}  \bar T_{\omega_m}    \mathcal S^{t}_{m+1,N}  T_m \cdots   T_1    \tau_1^{-1} s_1 \cdots s_m  f   
      \end{split}
\end{equation*}
\end{proof}

\end{appendix}


\begin{thebibliography}{99}

\bibitem{A} P. Alexandersson,  {\it Non-symmetric Macdonald polynomials and Demazure–Lusztig operators}, Séminaire Lotharingien de Combinatoire, 76 (2019). 

\bibitem{BF} T.~H.~Baker and P.~J.~Forrester {\it  A $q$-analogue of the type $A$ Dunkl Operator and Integral Kernel}, Int. Math. Res. Not. {\bf 14} (1997), 667–-686.  







\bibitem{BDLM2}
O. Blondeau-Fournier, P. Desrosiers, L. Lapointe and P. Mathieu,
{\it Macdonald polynomials in superspace as eigenfunctions of commuting operators}, Journal of Combinatorics {\bf 3}, no. 3, 495--562 (2012).






\bibitem{Che}
I. Cherednik, {\it Non-symmetric Macdonald polynomials}, Int. Math.  Res. Notices {\bf 10}  (1995) 483-515.














\bibitem{BG} B.G. Goodberry, {Partially-symmetric Macdonald polynomials}, Ph.D. Thesis, Virginia Polytechnic Institute and State University, 2022. 
  

\bibitem{Haiman}
M. Haiman, {\it Hilbert schemes, polygraphs, and the Macdonald
positivity conjecture}. J. Amer. Math. Soc., \textbf{14} (2001),
941--1006.

\bibitem{HHL}
J. Haglund, M. Haiman and N. Loehr, {\it A combinatorial formula for nonsymmetric Macdonald polynomials}, Amer. J. Math. {\bf 130} (2008), 359--383. 


\bibitem{L} L. Lapointe, {\it $m$-Symmetric functions, non-symmetric Macdonald polynomials and positivity conjectures}, arXiv:2206.05177.
  


  
\bibitem{M} I.~G. Macdonald, {Symmetric functions and Hall
polynomials}. 2nd edition, Clarendon Press, Oxford, 1995.

\bibitem{Mac1}
               {I.~G.~ Macdonald},
                {\it Affine Hecke algebras and orthogonal polynomials},
S\'eminaire Bourbaki 1994-95, expos\'e 797, p.  189-207. 
  
\bibitem{Mac2}
               {I.~G.~ Macdonald},
                {Affine Hecke algebras and orthogonal polynomials},
Cambridge Univ. Press (2003). 

\bibitem{Mar}
D. Marshall, {\it Symmetric and nonsymmetric Macdonald polynomials},
Ann. Comb. {\bf 3} (1999), 385--415.


\bibitem{MN} K. Mimachi and M. Noumi, {\it A reproducing kernel for non-symmetric Macdonald polynomials}, Duke Math. J. {\bf 91(3)} (1998), 621--634.



\bibitem{Stan} R.P. Stanley, Enumerative Combinatorics, vol. 2, Cambridge University Press, (1999).






\end{thebibliography}
\end{document}